\documentclass{amsart}
\usepackage{amssymb, amscd}
\usepackage{graphicx}
\usepackage{color}
\definecolor{darkgreen}{cmyk}{1,0,1,.2}
\definecolor{m}{rgb}{1,0.1,1}
\definecolor{green}{cmyk}{1,0,1,0}

\definecolor{test}{rgb}{1,0,0}
\definecolor{cmyk}{cmyk}{0,1,1,0}

\long\def\red#1{\textcolor {red}{#1}}

\long\def\green#1{\textcolor {green}{#1}}
\long\def\m#1{\textcolor {m}{#1}}

\newtheorem{Equation}{}[section]

\newtheorem{corollary}[Equation]{Corollary}
\newtheorem{definition}[Equation]{Definition}

\newtheorem{lemma}[Equation]{Lemma}
\newtheorem{proposition}[Equation]{Proposition}

\newtheorem{problem}[Equation]{Problem}
\newtheorem{remark}[Equation]{Remark}

\newtheorem{theorem}[Equation]{Theorem}

\def\ch{\operatorname{ch}}

\def\codim{\operatorname{codim}}
\def\Diff{\operatorname{Diff}}
\def\Mor{\operatorname{Mor}}
\def\Dom{\operatorname{Dom}}
\def\Aut{\operatorname{Aut}}
\def\SO{\operatorname{SO}}
\def\SU{\operatorname{SU}}

\def\reg{\operatorname{reg}}
\def\av{\operatorname{av}}

\def\Dom{\operatorname{Dom}}

\def\ind{\operatorname{ind}}
\def\End{\operatorname{End}}

\def\geo{\operatorname{geo}}

\def\Id{\operatorname{Id}}

\def\Hom{\operatorname{Hom}}

\def\Im{\operatorname{Im}}
\def\Ind{\operatorname{Ind}}

\def\Ker{\operatorname{Ker}}

\def\mod{\operatorname{mod}}

\def\Supp{\operatorname{Supp}}

\def\Spin{\operatorname{Spin}}

\def\Tr{\operatorname{Tr}}
\def\STr{\operatorname{STr}}
\def\tr{\operatorname{tr}}

\def\C{\mathbb C}
\def\D{\mathbb D}
\def\H{\mathbb H}
\def\N{\mathbb N}
\def\R{\mathbb R}
\def\S{\mathbb S}

\def\Z{\mathbb Z}
\def\B{\mathbb B}

\def\Q{\mathbb Q}

\def\cA{{\mathcal A}}
\def\maA{{\mathcal A}}
\def\maB{{\mathcal B}}

\def\maD{{\mathcal D}}

\def\maE{{\mathcal E}}

\def\maF{{\mathcal F}}
\def\maM{{\mathcal M}}
\def\cG{{\mathcal G}}

\def\maG{{\mathcal G}}
\def\cG{{\mathcal G}}

\def\cH{{\mathcal H}}

\def\maK{{\mathcal K}}
\def\maN{{\mathcal N}}

\def\maR{{\mathcal R}}

\def\maS{{\mathcal S}}
\def\maQ{{\mathcal Q}}

\def\maU{{\mathcal U}}
\def\cU{{\mathcal U}}

\def\what{\widehat}
\def\wtit{\widetilde}

\def\tD{\widetilde{D}} 
\def\tDp{\widetilde{D'}} 
\def\tF{\widetilde{F}} 
\def\tPhi{\widetilde{\Phi}}

\def\tH{\widetilde{H}}
\def\tX{\widetilde{X}}

\def\tG{\widetilde{G}}
 
\def\tF{\widetilde{F}} 
\def\tO{\widetilde{\Omega}}

\def\tM{\widetilde{M}}

\def\tS{\widetilde{S}}

\def\tT{\widetilde{T}}

\def\tX{\widetilde{X}}
\def\tZ{\widetilde{Z}}
\def\tg{\widetilde{g}}
\def\tQ{\widetilde{Q}}

\def\tY{\widetilde{Y}}
\def\tW{\widetilde{W}}
\def\tN{\widetilde{N}}

\def\tm{\widetilde{m}}

\def\dd{\displaystyle}
\def\pa{\partial}

\def\ep{\epsilon}

\begin{document}



\title[Relative $L^2$ index theorem] 
{The relative $L^2$ index theorem for Galois coverings }


\author[M-T. Benameur]{Moulay-Tahar Benameur}
\address{Institut Montpellierain Alexander Grothendieck, UMR 5149 du CNRS, Universit\'e de Montpellier}
\email{moulay.benameur@umontpellier.fr}

\thanks{MSC (2010) 53C12, 57R30, 53C27, 32Q10. }

\begin{abstract}
Gievn a Galois covering of complete spin manifolds where the base metric has PSC near infinity, we prove that  for small enough $\ep>0$, the $\ep$ spectral projection of the Dirac operator has finite trace in the Atiyah von Neumann algebra. This allows us to define the $L^2$ index  in the even case and we prove its compatibility with the Xie-Yu higher index and deduce  $L^2$ versions of the classical Gromov-Lawson relative index theorems. Finally, we briefly discuss some Gromov-Lawson $L^2$ invariants. 
 \end{abstract}
\maketitle
\tableofcontents

%

\section{Introduction}

In \cite{GromovLawson}, M. Gromov and B. Lawson introduced in their investigation of the topological space of metrics of positive scalar curvature on a given smooth manifold $M$, the notion of complete metrics having positive scalar curvature (PSC in the sequel) near infinity. They proved in particular that the Dirac operator $D_g$ associated with  a  spin structure  which is compatible with such metric has finite dimensional kernel. This allowed  them to introduce an integer invariant for such metrics on even dimensional spin manifolds, which is the Fredholm index of the Dirac operator $D_g^+$ acting from the positive half-spinors to the negative ones, exactly as for closed spin manifolds, we denote this invariant by $\Spin (M, g)$.  Despite the closed case, the invariant $\Spin (M, g)$ depends on $g$ and is not a topological invariant, although it does not depend on perturbations of the geometric data over compact subspaces. An already interesting class of examples where such invariant plays a significant part corresponds to the ``small'' class of complete manifolds with a riemannian cylindrical end, where the invariant $\Spin (M, g)$ can be shown to coincide with  the Atiyah-Patodi-Singer index for a corresponding compact spin manifold with boundary \cite{APS1}. This observation allows to relate $\Spin (M, g)$ with some standard spectral invariants but convinces as well that there will certainly be no simple formula for $\Spin (M, g)$ in the general case.   The relative index theorem remedies for this by giving a formula for the difference of such indices when one has identifications near infinity. Gromov and Lawson actually  proved as well their so-called  $\Phi$-relative index theorem  \cite{GromovLawson} that we briefly review below, see also \cite{LawsonMichelsohn}. 

Assume first that $E$ is some hermitian bundle, over the complete spin riemannian manifold $(M, g)$ with $g$ having PSC near infinity, which has a connection $\nabla^E$ which is flat near infinity, then the twisted Dirac operator $D_g\otimes E$ has again finite dimensional kernel and in the even case, the index of $D_g^+\otimes E$, denoted $\Spin (M, g; E)$ is well defined. The simplest statement of the relative index theorem is  the following formula where $K$ is any compact subspace of $M$ such that off $K$, $g$ has PSC and $E$ is flat:
$$
\Spin (M, g; E) - \dim (E)\, \Spin (M, g) =  \int_K \widehat{A} (M, g) \ch_{>0} (\nabla^E),
$$
We have denoted as usual $\widehat{A} (M, g)$ the $\widehat{A}$-polynomial in the Pontryagin forms of the Levi-Civita connection associated with $g$, and $\ch_{>0} (\nabla^E):= \ch (\nabla^E) - \dim (E)$ where $\ch (\nabla^E)$ denotes the Chern character polynomial in the curvature of  the hermitian connection $\nabla^E$.   An obvious corollary is that the LHS is an obstruction to extending $\nabla^E$ to a flat connection on $M$. The general relative index theorem gives a similar formula for the difference of the indices of two generalized Dirac operators on different smooth complete manifolds, such that the operators agree near infinity via some isometry. In this case the individual indices don't make sense in general, however the relative index (difference of the virtual two indices)  still makes sense by replacing each manifold by some closed manifold obtained by chopping off the isometric open subspaces and attaching the same compact manifold with boundary to each of them, and then taking the difference of the resulting indices, see \cite{GromovLawson}[page 119]. Then a similar relative index theorem holds by using well chosen parametrices for the generalized Dirac operators. 
Finally, they also gave another  version: the so-called Gromov-Lawson $\Phi$-relative index theorem. Here,  one assumes that both generalized Dirac operators are invertible near infinity so that the individual indices do make sense, but they don't assume anymore that the identification of the operators holds off some compact space but only on some union $\Phi$ of  connected components of a neighborhood of infinity. The $\Phi$-relative index is then defined similarly, but now one has to chop off the isometric parts and just glue the two remaining manifolds so that the resulting  generalized Dirac operator is invertible near infinity, therefore has a well defined index. The index formula in this $\Phi$-relative case, identifies the index of this latter resulting generalized Dirac operator with the difference of the indices of the generalized Dirac operators on the original manifolds. 

All these relative index theorems played a crucial role in   their study of enlargeable manifolds and  the moduli  space of metrics of PSC. Gromov and Lawson could for instance deduce obstruction criteria in terms of (non-compact) enlargeability of closed manifolds, but also some similar criteria for complete non compact manifolds, see the nice summary in \cite{LawsonMichelsohn}. Already in the APS cylindrical ends case, it was for instance combined with some topological results of Milnor,  to recover  in \cite{GromovLawson} other non-trivial corollaries. Let us mention for instance their proof of  the infiniteness of the number of connected components of the space of PSC metrics on the $7$-sphere $\S^7$. It is worthpointing out  though that in the cylindrical ends case, all the relative index theorems reviewed above become corollaries of the APS index theorem for manifolds with boundary, which provides a more precise answer since it gives an index formula for each individual manifold. \\

The present paper is devoted to the $L^2$ version of the Gromov-Lawson relative index theory together with some applications. Our setting will be the category of riemannian Galois $\Gamma$-coverings $\pi: \tM\to M$ over the complete riemannian manifold $M$, with  $\Gamma$ being a discrete countable group and we proceed now to describe some results that we obtained. Given a generalized Dirac operator over $M$, recall that $D^2=\nabla^*\nabla + \maR$ with an explicit zero-th order operator $\maR$. We then consider its lift $\tD$  acting on the $L^2$-sections of the generalized spinors $\tS$ over $\tM$. The operator $\tD$ is affiliated with the Atiyah semi-finite von Neumann algebra $\maM$ of $\Gamma$-invariant bounded operators on $L^2 (\tM, \tS)$. The Atiyah trace on $\maM$ is denoted $\tau$ and corresponds roughly speaking to integration over fundamental domains of the Galois covering.  For any $\ep\geq 0$, we may then consider the spectral projection $P_\ep$ of $\tD$ corresponding to the intervalle $[-\sqrt{\ep}, +\sqrt{\ep}]$, an element of $\maM$. Our first result is the following

\begin{theorem}\label{MainPrel}
Assume that there exists $\kappa_0 >0$ such that $\maR\geq \kappa_0\Id$ off some compact subspace of $M$, then there exists $\ep>0$ such that $P_\ep$ has finite $\tau$-trace. 
\end{theorem}

In particular, this proves in the even case that the $L^2$-index is well defined by
$$
\Ind_{(2)} (\tD^+) = \dim_\Gamma (\Ker \tD^+) - \dim_\Gamma (\Ker \tD^-).
$$
In the case where $M$ is a  spin manifold and $D=D_g$ is the spin-Dirac operator, this shows that if $g$ has PSC near infinity then the $L^2$-index is well defined. It will be denoted in this case $\Spin_{(2)} (\tM, g)$. When $M=N\times \R$ for a closed spin odd dimensional manifold $N$, we deduce the $L^2$ version of the Gromov-Lawson invariant $i(N; g, g')$ that we denote by $i_{(2)} (\tN; g, g')$ for any metrics $g$ and $g'$ with PSC. 

In \cite{XieYu}, Xie and Yu defined a higher relative index living in the $K$-theory of the group $C^*$-algebra and proved a higher version of the $\Phi$-relative index theorem. We have thus privileged to use the Xie-Yu approach to deduce our $L^2$ versions of the relative index theorems. Inorder to achieve this program, we were led to prove the compatibility of our $L^2$-index with the higher one through the usual regular trace on the $C^*$-algebra of $\Gamma$. Inorder to prove this compatibility result, we use results proved in \cite{BenameurRoy} together with the crucial property of finite $\tau$-trace of $P_\ep$, for some $\ep>0$. So our second result can be stated as follows

\begin{theorem}\label{CompatibilityPrel}
Assume that $M$ is even dimensional and denote by $\maD_r$ the regular Michschenko-Fomenko Dirac operator associated with the generalized Dirac operator $D$ and by $\Ind(\maD^+_r)$ its Xie-Yu higher index in $K_0(C_r^*\Gamma)$. If $\tau^{\reg}_{*}: K_0(C^*_r\Gamma) \to \R$ is the group morphism induced by the regular trace $\tau^{\reg}$, then 
$$
\tau^{\reg}_* \left( \Ind(\maD^+_r)\right) = \Ind_{(2)} (\tD^+).
$$
\end{theorem}

Applying this theorem together with a construction of the $\Phi$-relative $L^2$-index, we could deduce the following $\Phi$-relative $L^2$ index theorem:

\begin{theorem}\label{PhiRelativePrel}
Assume that we have two Galois $\Gamma$-coverings $\tM\to M$ and $\tM'\to M'$ with generalized Dirac operators which both satisfy the invertibility condition near infinity $\maR\geq \kappa_0 \Id$ as in \cite{GromovLawson}, and assume furthermore that there exist unions $\Phi$ and $\Phi'$ of connected components of a neighborhood of infinity in $M$  and $M'$ respectively which are identified together with their $\Gamma$-coverings so that the generalized Dirac operators $D$ and $D'$ over $\Phi$ and $\Phi'$ are conjugate. Then the  $\Phi$-relative $L^2$-index $\Ind_{(2)} (\tD, \tDp; \Phi)$ is a well defined real number and we have the $\Phi$-relative $L^2$ index formula
$$
\Ind_{(2)} (\tD, \tDp; \Phi) = \Ind_{(2)} (\tD^+) - \Ind_{(2)} ({\tDp}^+). 
$$
\end{theorem}

In the special case where the open subspaces $\Phi$ and $\Phi'$ have compact complements in $M$ and $M'$, no need of any invertibility condition near infinity to define the $L^2$ relative index $\Ind_{(2)} (\tD, \tDp)$ for any generalized Dirac operators which are conjugate near infinity. In this case, we could deduce the following Atiyah  relative $L^2$ index formula which extends the classical result proved in \cite{AtiyahCovering}:

\begin{theorem}
Let $\tM\to M$ and $\tM'\to M'$ be two Galois $\Gamma$-coverings with generalized Dirac operators $D$ and $D'$ on $M$ and $M'$ respectively.  Assume that the coverings are identified near infinity so that the lifts $\tD$ and $\tDp$ which are $\Gamma$-equivariantly conjugate near infinity. Then the relative $L^2$ index $\Ind_{(2)} (\tD, \tDp)$ is well defined and agrees with the Gromov-Lawson relative index of the operators $D$ and $D'$ on $M$, i.e.
$$
\Ind_{(2)} (\tD, \tDp) = \Ind (D, D').
$$
\end{theorem}

When $\Gamma$ is torsion free,  we have the following partial but more precise result. 

\begin{theorem}
Assume that $\maR$ is uniformly positive near infinity and that $\Gamma$ is torsion free with rationally onto maximal Baum-Connes assembly map $K_0(B\Gamma)\to K_0(C^*_m\Gamma)$. Then
$$
\Ind_{(2} (\tD^+) =\Ind (D^+).
$$
\end{theorem}

In the spin case, we obtain the equality $
\Spin_{(2)} (\tM, g) = \Spin (M, g),$
whenever $g$ has uniform PSC near infinity and $\Gamma$ is torsion free with rationally onto maximal Baum-Connes map.
A corollary of this theorem is the corresponding statement for the APS $L^2$ index for coverings with boundaries. See Theorem \ref{AtiyahAPS}. \\

Our $L^2$ index allows to consider for instance the generalized Cheeger-Gromov invariant 
$$
\kappa_\Gamma (M, g) := \Spin_{(2)} (\tM, g) - \Spin (M, g),
$$
which is expected to be trivial for torsion free groups but provides an interesting invariant in general. In the cylindrical ends case considered by Atiyah-Patodi-Singer out of a compact spin manifold $X$ with boundary $Y$ having PSC, this invariant corresponds to the Cheeger-Gromov $L^2$ rho invariant $\rho_\Gamma (Y, g_Y)$, but in general we can only show that it only depends on the geometry near infinity with, so far, no explicit spectral expression for instance.
We recover for instance the well known fact that the Cheeger-Gromov $\rho$  invariant induces a map on the moduli space of connected components of the space of PSC metrics on any closed odd dimensional manifold modulo orientation preserving diffeomorphisms. This $\kappa$ invariant for $M=N\times \R$ as in \cite{GromovLawson} with a Galois $\Gamma$-covering $\tN\to N$, yields the invariant
$$
\kappa_\Gamma (N; g, g') :=  i_{(2)} (\tN; g, g') - i (N; g, g'). 
$$
By the APS theorem this invariant is again just a difference of Cheeger-Gromov $L^2$ rho invariants for the given metrics. When $N$ has trivial real Pontryagin classes, one can extend similarly some results from \cite{KreckStolz}. These and other applications will be carried out in a forthcoming paper. A short appendix gives a different proof of the finite $\tau$-trace of the spectral projection $P_\ep$ now using a von Neumann version of the Rellich lemma for which we provide an independent  proof.


\medskip

{\em{Acknowledgements.}} The author would like to thank B. Ammann, P. Antonini, A. Carey,  P. Carrillo-Rouse, J. Heitsch, V. Mathai and G. Yu  for several discussions.

\medskip

\section{Dirac operators and Atiyah's von Neumann algebra}\label{Preliminaries}

Let $\Gamma$ be a countable finitely generated discrete group and let $\pi: \tM \to M$ be a Galois $\Gamma$-covering over the smooth complete  riemannian manifold $(M, g)$. We endow $\tM$ with the lifted $\Gamma$-invariant metric $\tg$, so that it is also a complete riemannian manifold and $\pi$ is an isometric covering. We fix a hermitian bundle $S$ of generalized spinors over $M$, in the sense of \cite{GromovLawson}[Section I]. Its pull-back to $\tM$ is denoted $\tS$. We fix corresponding  Lebesgue class measures on $M$ and $\tM$ that we denote by $dm$ or $d\tm$ respectively. We then consider the Hilbert spaces of $L^2$-sections of these spinor bundles $S$ and $\tS$ that we denote by $L^2(M, S)$ and $L^2(\tM, \tS)$  respectively. Notice that $L^2 (\tM, \tS)$ is a Hilbert space which is endowed with the unitary representation of $\Gamma$. Let $F\subset \tM$ be an open fundamental domain for the Galois cover. This means for us that  $\gamma F\cap F=\emptyset$ for  any $g\neq e$, and that the collection $(g\overline{F})_{\gamma\in \Gamma}$ is a locally finite cover of $\tM$. We may assume as well that $F$ equals the interior of its closure $\overline{F}$ if needed.  It is then worth pointing out that the subspace $\tM \smallsetminus \cup_{\gamma\in \Gamma} g F$ is $d\tm$-negligible. This choice of fundamental domain $F$ allows to identify the Hilbert $\Gamma$-space  $L^2 (\tM, \tS)$ with the Hilbert $\Gamma$-space $\ell^2\Gamma \otimes L^2(F, \tS)$ where $\Gamma$ acts trivially on $L^2 (F, \tS)\simeq L^2(M, S)$ and by the left regular representation on $\ell^2\Gamma$.  See \cite{AtiyahCovering} for more details on this standard construction.

The von Neumann algebra of bounded $\Gamma$-invariant operators on $L^2 (\tM, \tS)$ will be denote by $\maM$. The elements of $\maM$ are bounded operators $T$ acting on the Hilbert space $L^2 (\tM, \tS)$ which commute with the unitary representation of $\Gamma$. This is a semi-finite von Neumann algebra which is isomorphic to $B(L^2(F, \tS))\otimes \maN \Gamma$ where $\maN\Gamma$ is the regular von Neumann algebra associated with $\Gamma$, say  the weak closure of the left regular representation in $\ell^2\Gamma$. There is a faithful normal semi-finite positive trace $\tau$ on $\maM$ corresponding through the isomorphism $\maM\simeq B(L^2(F, \tS))\otimes \maN \Gamma$ to the trace $\Tr\otimes \tau_e$ where $\tau_e$ is the finite trace on $\maN\Gamma$ given by $\theta\mapsto <\theta(\delta_e), \delta_e>$ with $\delta_e$ being the characteristic function at the neutral element  $e$. $\Tr$ denotes as usual the trace of operators on the Hilbert space $L^2(F, \tS)$. It can also be defined directly as follows, see \cite{AtiyahCovering}.  Denote by $\chi_F$ the characteristic function of the fundamental domain $F$, i.e. the Borel function which equals $1$ on $F$ and $0$ on its complement. Then multiplication by the function $\chi_F$ yields a projection $M_{\chi_F}: L^2(\tM, \tS)\to L^2(\tM, \tS)$ whose image can be identified with $L^2 (F, \tS)$. 
\begin{definition}\cite{AtiyahCovering}\
For any non-negative operator $T\in \maM$,  we set
$$
\tau (T):= \Tr (M_{\chi_F}\circ T \circ M_{\chi_F}) \quad \in [0, +\infty].
$$
\end{definition}

That $\tau$ is a faithful normal semi-finite positive trace on $\maM$ which does not depend on the choice of fundamental domain $F$ is straightforward, see again \cite{AtiyahCovering}. Any operator in $\maM$ which corresponds to a finite sum of elementary tensors $A_i\otimes \theta_i$ with  trace class operators $A_i$ on the Hilbert space $L^2(F, \tS)$ and with $\theta_i\in \maN\Gamma$, will obviously have finite $\tau$-trace. Given a Hilbert $\Gamma$-subspace $H$ of $L^2(\tM, \tS)$, we denote by $P_H$ the $\Gamma$-invariant orthogonal projection  onto $H$, a non-negative idempotent in $\maM$. The $\Gamma$-dimension of $H$, denoted $\dim_\Gamma (H)$, is defined as the trace of the projection $P_H\in \maM$, i.e. 
$$
\dim_\Gamma (H) := \tau (P_H) \quad \in [0, +\infty].
$$
We can define in the same way the $\Gamma$-dimension of any Hilbert $\Gamma$-space which is unitarily equivalent to a Hilbert $\Gamma$-subspace of $L^2(\tM, \tS)$. In particular, given a Hilbert subspace $H_0$ of $L^2(F, \tS)$, the $\Gamma$-dimension of $\ell^2\Gamma\otimes H_0$ is well defined and coincides with the usual dimension of $H_0$.  When $H_0$ is finite dimensional, such Hilbert $\Gamma$-space what is  called in   \cite{Lueck} a {\underline{finitely generated}} Hilbert $\Gamma$-space. \\

When $M$ is even dimensional, the generalized spin bundle $S$ will be assumed to have a $\Z_2$-grading and to split into $S=S^+\oplus S^-$, we then have the corresponding $\Gamma$-equivariant splitting   $\pi^{*} (S)=\tS=\tS^+\oplus \tS^-$.  Let $D$ be  a generalized Dirac operator on $M$ acting on the smooth sections of  $S$, and denote by $\tD$ its $\Gamma$-invariant pull-back to the generalized Dirac operator on $\tM$, acting on the sectiojns of $\tS$. 
Recall that $\tD$ is a $\Gamma$-invariant first order differential operator which is odd for the grading, i.e.
$$
\tD = \left(\begin{array}{cc} 0 & \tD^-\\ \tD^+ & 0\end{array}\right).
$$
The orthogonal connection on $S$ is denoted $\nabla$ and its pull-back connection to $\tS$ is denoted $\widetilde\nabla$. Recall that if $(e_k)_k$ is a local orthonormal basis of tangent vectors then the operator $\nabla^*\nabla$ is given by the formula
$$
\nabla^*\nabla = -\sum_k \left( (\nabla_{e_k})^2 - \nabla_{\nabla_{e_k}e_k}\right)\text{ and similarly for }\widetilde{\nabla}^*\widetilde{\nabla}.
$$
Recall as well that for smooth sections $\sigma, \sigma'$ where at least one of them is compactly supported, one has
$$
\langle \nabla^*\nabla \sigma, \sigma'\rangle = \langle\nabla\sigma, \nabla\sigma'\rangle.
$$
For $s\geq 0$,  the Sobolev space $L^2_s (M, S)$ (resp. $L^2_s (\tM, \tS)$) is defined  as the completion of the space $C_c^\infty (M, S)$ (resp. $C_c^\infty (\tM, \tS)$) of smooth compactly supported sections, for the Sobolev $L^2$-norm
$$
\vert\vert \sigma\vert\vert_s^2 := \sum_{j=0}^s \vert\vert \nabla^j \sigma\vert\vert^2, \quad \sigma\in C_c^\infty (M, S),
$$
and similarly for $\tM$ using the $\Gamma$-invariant pull-back connection $\widetilde{\nabla}$. In particular, $L^2_0 (M, S)= L^2 (M, S)$ and $L^2_0 (\tM, \tS)= L^2 (\tM, \tS)$. We denote by $\Omega^S$ the curvature tensor of the connection $\nabla$ on $S$ given by
$$
\Omega^S (X_1, X_2)= [\nabla_{X_1}, \nabla_{X_2}] - \nabla_{[X_1, X_2]}\text{ for any vector fields } X_1, X_2.
$$
The pull-back $\widetilde{\Omega^S}$ of $\Omega^S$ to $\tM$ is nothing but the curvature tensor of the pull-back conneciton $\widetilde{\nabla}$ on $\tS$.  The following theorem is classical, see also \cite{GromovLawson}[Propositions 2.4 $\&$ 2.5]:

\begin{theorem}\label{Lichne}\cite{Lichne}\
The operator $\nabla^*\nabla$ is an essentially self-adjoint operator. Moreover, for $u\in L^2(\tM, \tS)$, we have 
$$
\nabla^*\nabla (u) = 0 \Longleftrightarrow \nabla u=0 \text{ (i.e. }u\text{ is parallel)}.
$$
Moreover, the following generalized Lichnerowicz local formula holds 
$$
D^2 = \nabla^*\nabla  + \maR\text{ with }\maR:=\frac{1}{2} \sum_{i, j} e_i \cdot e_j\cdot \Omega^S (e_i, e_j),
$$ 
with $(e_i)_i$ being any local orthonormal frame on $\tM$. 
\end{theorem}

\medskip

The corresponding relation holds obviously on $\tM$ since it satisfies the same riemannian conditions satisfies by $M$. In the sequel, we shall sometimes also denote by $\maR$ the pull-back operator $ \widetilde{\maR}$ which is defined by the same formula. A standard calculation  shows that when $D$ is the spin-Dirac operator associated with a fixed spin structure associated with the $\SO$ bundle corresponding to the metric $g$ on $M$, then $
\maR = \frac{\kappa}{4} \Id_S,$
where $\kappa$ is the scalar curvature function of $g$.


\medskip

\begin{proposition}\label{Self-adjoint}\ 
Under the previous assumptions, we have:
\begin{enumerate}
\item $\tD:C_c^\infty (\tM, \tS)\to C_c^\infty (\tM, \tS)$ is  essentially self-adjoint whose self-adjoint extension has domain contained in $L^2_1 (\tM, \tS)$;
\item If the pointwise norm $\vert\maR\vert$ of the zero-th order operator $\maR$  is uniformly bounded on $M$, then the domain of the self-adjoint extension of $\tD$ is exactly $L^2_1 (\tM, \tS)$ and $\tD$ gives  a bounded operator from $L^2_1 (\tM, \tS)$ to $L^2 (\tM, \tS)$;
\item $\tD^2:C_c^\infty (\tM, \tS)\to C_c^\infty (\tM, \tS)$ is essentially self-adjoint and its self-adjoint extension is a non-negative operator. 
\item  The  kernels of the operators $\tD$ and $\tD^2$ on $L^2(\tM, \tS)$ do coincide. Moreover, they are composed of smooth sections. 
\end{enumerate}
\end{proposition}

\begin{proof}\
Most of the items are standard results for generalized Dirac operators on   complete  manifolds, see for instance \cite{GromovLawson}[Theorems 1.17, 1.23] and also \cite{LawsonMichelsohn}[Theorem II.5.7]. The coincidence of the domains with the Sobolev spaces under the condition that $\maR$ is bounded is a consequence of the Lichnerowicz formula, see \cite{GromovLawson}[Theorem 2.8]. For the second part of the last item, notice that if $U$ is a relatively compact open subspace of $\tM$, then any $L^2$ harmonic section restricts to a harmonic $L^2$ section in $L^2(U, \tS)$.
\end{proof}

Given any open relatively compact subspace $\Omega$ of $M$ and its inverse image $\tO$ in $\tM$,  the manifold  $\tO$ has bounded geometry and so is  the  restricted bundle  $\tS\vert_{\tO}$. Therefore, classical elliptic estimates apply and we obtain for instance the existence for any $s\in \Z$ of a constant $C_s>0$ such that for any $u\in L^2_{s} (\tO, \tS)$:
$$
\vert\vert u\vert\vert_{L^2_s(\tO, \tS)}  \leq C_s (\vert\vert u\vert\vert_{L^2_{s-1}(\tO, \tS)} + \vert\vert \tD(u)\vert\vert_{L^2_{s-1}(\tO, \tS)})
$$
Here again the Sobolev spaces are defined as for $\tM$ completing the smooth compactly supported sections on $\tO$.  Using Proposition \ref{Self-adjoint}, we hence deduce that the norms $\vert\vert \bullet\vert\vert_{L^2_s(\tO, \tS)}$ and $\vert\vert \bullet \vert\vert_{L^2_{s-1}(\tO, \tS)} + \vert\vert \tD(\bullet)\vert\vert_{L^2_{s-1}(\tO, \tS)}$ are equivalent on $L^2_{s} (\tO, \tS)$.

\medskip

\section{Von Neumann  trace of the spectral $\ep$-projection}

Let as before $(M, g)$ be a complete riemannian manifold and assume now that  there exists a compact subspace  $K\subset M$ such that the zero-th order curvature operator  $\maR$ satisfies the relation 
$$
\maR\vert_{M\smallsetminus K}  \geq \kappa_0 \Id \text{ for some constant }\kappa_0>0.
$$
We shall refer to this condition as  (uniform) {\em{invertibility near infinity}}. If $D$ is a given generalized Dirac operator, then any twist of $D$ by a hermitian connection on an extra bundle $E\to M$ yields a new generalized Dirac operator $D^E$ acting on the sections of the generalized spin bundle $S\otimes E$. The operator $\maR^E$ corresponding to $D^E$ is then given by
$$
\maR^E = \maR +  \frac{1}{2} \sum_{i, j} (e_i\cdot e_j) \otimes \Omega^E,
$$
where $\Omega^E$ is the curvature tensor of $E$. In particular, if $D$ is   invertible near infinity and $E$ is  $\ep$-almost flat for a sufficiently small $\ep$, then  $D^E$ will also be invertible near infinity.  
In the spin case, the spin-Dirac operator is invertible near infinity if and only if the scalar curvature function $\kappa$ satisfies
$$
\frac{\kappa}{4} \geq \kappa_0\; \text{ off some compact subspace}\; K\text{ for some constant }\kappa_0>0.
$$
 We are now in position to state the main theorem of this section. 

\medskip

\begin{theorem}\label{Main}
Assume  that the generalized Dirac operator $D$ acting on the sections of $S$ over the complete riemannian manifold $M$ is uniformly invertible  near infinity. Then, for any Galois $\Gamma$-cover $\tM\to M$, there  exists $\ep >0$ such that the spectral projection $P_\ep$ of the $\Gamma$-invariant Laplacian $\tD^2$,  associated with the intervalle $[0, \ep]$,  has finite $\Gamma$-dimensional range, i.e. $
\tau (P_\ep) <+\infty$. 
\end{theorem}

\medskip

To emphasize the spin-Dirac operator associated with the metric $g$ when there is an associated spin structure, we shall sometimes denote it by $D_g$ and by $\tD_g$ for the lift to $\tM$. The spectral projection  $P_\ep$ is an element of Atiyah's von Neumann algebra $\maM=B(L^2(\tM, \tS))^\Gamma$. When the Galois cover is the trivial one corresponding to $\Gamma=\{e\}$, this theorem is due to Gromov and Lawson who prove moreover that in this case there is a gap  in the spectrum near $0$ \cite{GromovLawson}. They hence deduced that the Green operator associated with $D$ is $L^2$-bounded \cite{GromovLawson}. For $\tD$ this is not true for general infinite $\Gamma$ as can  be checked in simple examples already with $M$ compact.

\begin{remark}
It is easy to see that the kernel projection $P_\ep$  has a smooth  Schwartz kernel $k_\ep (\tm, \tm')\in \Hom (S_{m'}, S_m)$. Theorem \ref{Main} then implies that  the integral $
\int_F \tr (k_\ep (\tm, \tm)) d\tm$, over any fundamental domain $F$ of the Galois cover $\tM\to M$, is finite and coincides with  $\tau (P_\ep)$. It is worth pointing out that the finiteness of this integral  can be proved directly but, even when $\Gamma$ is trivial,  it does not suffice to deduce Theorem \ref{Main}. 
\end{remark}

We shall give the proof of Theorem \ref{Main} in Section \ref{Proof}, and we also show in Appendix \ref{Rellich} that a type II Rellich lemma also allows to deduce another proof of this theorem.  
We can deduce the following

\medskip

\begin{theorem}
Under the assumptions of Theorem \ref{Main} and when $M$ is even dimensional, the $\Gamma$-invariant generalized Dirac operator $\tD^+$  acting from $L^2 (\tM, \tS^+)$ to $L^2 (\tM, \tS^-)$ has a well defined $L^2$ index
$$
\Ind_{(2)} (\tD^+) := \dim_\Gamma (\Ker (\tD^+)) -  \dim_\Gamma (\Ker (\tD^-))  \quad \in \R.
$$
In the spin case with the metric $g$ having PSC near infinity, this $L^2$-index is denoted $\Spin_{(2)} (\tM, g)$ and is called the $L^2$-genus of the riemannian Galois cover.
\end{theorem}

\medskip

\begin{proof}
It is clear from the very definition of the von Neumann trace $\tau$ that it can be restricted to operators on $\tS^\pm$ and that for any diagonal non-negative operator $T=\left (\begin{array}{cc}T^+ & 0\\ 0 & T^-\end{array}\right)$, we have $\tau (T) = \tau (T^+) + \tau (T^-)$. We apply Proposition \ref{Self-adjoint} to deduce that 
$$
\Ker(\tD)=\Ker (\tD^2) =  \Ker (\tD^-\tD^+)\oplus \Ker (\tD^+ \tD^-).
$$ 
Therefore, 
$$
+\infty > \dim_\Gamma (\Ker(\tD)) = \dim_\Gamma (\Ker (\tD^2)) = \dim_\Gamma (\Ker(\tD^-\tD^+)) + \dim_\Gamma (\Ker(\tD^+\tD^-)),
$$
and henceforth $ \dim_\Gamma (\Ker(\tD^-\tD^+))$ and $\dim_\Gamma (\Ker(\tD^+\tD^-))$ are both finite non-negative real numbers which similarly coincide respectively with $ \dim_\Gamma (\Ker(\tD^+))$ and $\dim_\Gamma (\Ker(\tD^-))$. 
\end{proof}

When $M$ is a closed  manifold, the invertibility near infinity condition  disappears and the  $L^2$-index $\Ind_{(2)} (\tD)$ is known to belong to the range of the additive $K$-theory map associated with the regular trace  on the reduced $C^*$-algebra $C^*_r\Gamma$. Moreover, by the covering Atiyah theorem \cite{AtiyahCovering}, $\Ind_{(2)} (\tD)$ is then an integer which coincides with  the index $\Ind (D)$ of the generalized Dirac operator on the base manifold $M$, a topological invariant. 

\medskip

\section{Proof of Theorem \ref{Main}}\label{Proof}

We denote for any Borel subspace $A$ of $\tM$ by $\vert\vert\bullet\vert\vert_A$ the $L^2$ norm of the restriction to $A$. For a given $\phi\in L^2(\tM, \tS)$ we denote by $\beta_\gamma(\phi)$ the non-negative real number $\vert\vert \phi\vert\vert^2_{\gamma F}$, so that 
$$
\vert\vert \phi\vert\vert^2 = \sum_{\gamma\in \Gamma} \beta_\gamma(\phi).
$$
We make the assumptions of Theorem \ref{Main} and our goal here is to give a direct proof inspired from the Gromov-Lawson proof for trivial $\Gamma$.  Recall that $F$ denotes an open  fundamental domain for the Galois covering $\pi:\tM\to M$ of complete  manifolds. We shall first concentrate on the kernel projection $P=P_0$ and give a direct proof of its finite $\tau$-trace which  provides, after careful inspection,  a quantitative estimate depending on the covering  and of the constant $\kappa_0$ as well as the uniform norm of $\maR$ on the compact space $K$. The proof for $P_\ep$ with $\ep>0$ will be given right after using a parametrix and hence a different method which can also be applied to deduce the case $\ep=0$ without providing precise estimates. 
Since $K$ is compact, we know that there exists a constant $c>0$ such  $\maR \geq -c \Id$ in restriction to  $\pi^{-1} (K)$.

\begin{lemma}\label{PSCestimate}
For  any $\phi\in \Ker (\tD)$ and any $\gamma\in \Gamma$, the following estimate holds:
$$
\vert\vert \phi\vert\vert_{\gamma F\cap \pi^{-1} (K)} \geq \frac{\kappa_0}{\kappa_0+c} \times \sqrt{\beta_\gamma(\phi)}.
$$
\end{lemma}

\begin{proof}
Let  $\phi\in \Ker(\tD)$ be a given non trivial section and let us fix some $\gamma\in \Gamma$ such that $\beta_\gamma(\phi) >0$.  Then  we can assume that $\beta_\gamma(\phi)=1$ and  apply the generalized Lichnerowicz formula to $\phi$ to deduce
$$
0 = \int_{\gamma F} \vert (\tD \phi)(\tm)\vert_{S_m}^2  d\tm  = \int_{\gamma F} \vert (\nabla \phi)(\tm)\vert_{S_m\otimes T^*_mM}^2  d\tm  + \int_{\gamma F}  \langle \maR u, u\rangle_{S_m} d\tm \geq \int_{\gamma F}  \langle \maR u, u\rangle_{S_m} d\tm.
$$
But 
$$
\int_{\gamma F} \langle \maR u, u\rangle_{S_m} d\tm = \int_{\gamma F\cap\pi^{-1} (K)}  \langle \maR u, u\rangle_{S_m} d\tm +   \int_{\gamma F\smallsetminus \pi^{-1} (K)} \langle \maR u, u\rangle_{S_m} d\tm.
$$
On the other hand for  $m\in M\smallsetminus K$ we have $\maR_m \geq \kappa_0 \Id_{S_m}$, therefore
$$
\int_{\gamma F\smallsetminus \pi^{-1} (K)}  \langle \maR u, u\rangle_{S_m} d\tm \geq  \kappa_0 (1 - \vert\vert u\vert\vert^2_{\gamma F\cap \pi^{-1} (K)}),
$$
 and we also have
$$
 \int_{\gamma F\cap\pi^{-1} (K)}  \langle \maR u, u\rangle_{S_m} d\tm \geq -c \vert\vert u \vert\vert^2_{\gamma F\cap\pi^{-1} (K)}.
$$
Thus we deduce
$$
0 \geq \kappa_0 (1 - \vert\vert u\vert\vert^2_{\gamma F\cap \pi^{-1} (K)}) - c \vert\vert u\vert\vert^2_{\gamma F\cap \pi^{-1} (K)} = \kappa_0 - (\kappa_0+c)  \vert\vert u\vert\vert^2_{\gamma F\cap \pi^{-1} (K)},
$$
and hence the conclusion.
\end{proof}

We denote by $\vert\vert\bullet\vert\vert_{C^1, A}$, for an open subspace $A$ of $\tM$,  the uniform $C^1$-norm over $A$. 

\begin{proposition}\label{C1estimate}
Under the previous assumptions, and for any open relatively compact subspace $L$ of $M$, there exists  a constant $C= C(L)>0$ such that for any $\phi\in \Ker (\tD)$ and any $\gamma\in \Gamma$, we have
$$
\vert \vert\phi\vert\vert_{C^1, \gamma F\cap\pi^{-1}(L)} \leq C \times \sqrt{\beta_\gamma(\phi)}.
$$
\end{proposition}

\begin{proof}
Choose a compact subspace $K'$ which is contained in $\pi^{-1}L\cap \gamma F$. Let $\Omega$ be an open neighborhood of $K'$ in the open subspace $\gamma F$ which is relatively compact. Fix a smooth compactly supported function $\chi$ on $\Omega$ which equals $1$ on $K'$. For any $\phi\in \Ker (\tD)$, the uniform elliptic estimate recalled above insures that  for any $s\in \N$, there exists of a constant $C'_s$ such that 
$$
\vert\vert \phi\vert\vert_{L^2_s (K', \tS)} \leq C'_s \left( \vert\vert \chi\phi\vert\vert_{L^2_{s-1} (\Omega, \tS)} + \vert\vert \tD(\chi\phi)\vert\vert_{L^2_{s-1} (\Omega, \tS)}\right).
$$
Since $\phi\in \Ker (\tD)$, $\tD(\chi\phi) = [\tD, M_\chi] (\phi)$ where the operator $[\tD, M_\chi]$ is  a zero-th order differential operator with compactly supported coefficients inside $\Omega$, and hence there exists  a constant $C''_s$ such that
$$
\vert\vert \tD(\chi\phi)\vert\vert_{L^2_{s-1} (\Omega, \tS)} \leq C''_s \vert\vert \phi\vert\vert_{L^2_{s-1} (\Omega, \tS)}.
$$
Therefore,  there exists of a constant $C_s>0$ such that $
\vert\vert \phi\vert\vert_{L^2_s (K', \tS)} \leq C_s \vert\vert u\vert\vert_{L^2_{s-1} (\Omega, \tS)}.$
Notice that the constants $C_s, C'_s$ and $C''_s$ don't depend neither on the group element $\gamma$ nor on $\phi$.

Let now  $(\Omega_n)_{1\leq n\leq N}$ be a finite collection of open relatively compacts subspaces of $\gamma F$ such that
$$
K' \subset \Omega_1 \text{ and } \overline{\Omega_j} \subset \Omega_{j+1} \text{ for }1\leq j\leq N-1.
$$
Then we may apply the previous argument inductively to deduce the existence of a constant $C_s (N)$ such that for any $\gamma\in \Gamma$ and any $\phi\in \Ker (\tD)$:
$$
\vert\vert \phi\vert\vert_{L^2_s (K', \tS)} \leq  C_s(N) \vert\vert u\vert\vert_{L^2_{s-N} (\gamma F, \tS)}.
$$
Taking $s=N$ we get  a constant $C(N)>0$ such that for any $\gamma\in \Gamma$ and any $\phi\in \Ker (\tD)$:
$$
\vert\vert \phi\vert\vert_{L^2_{N} (K', \tS)} \leq C (N) \sqrt{\beta_\gamma(\phi)}.
$$
Taking the supremum over such subspaces $K'$ of $\pi^{-1} (L)\cap \gamma F$, we deduce by the Beppo-Levi argument and since our measure is regular (being a Lebesgue-class measure),
$$
\vert\vert \phi\vert\vert_{L^2_{N} (\pi^{-1}(L)\cap \gamma F, \tS)} \leq C (N) \sqrt{\beta_\gamma(\phi)}.
$$
Since $L$ is relatively compact, the Sobolev estimate implies that  for $N$ large enough  ($N> \dim (M)+1$), there exists a constant $C'$ such that
$$
\vert \vert\phi \vert\vert_{C^1, \gamma F\cap\pi^{-1}(L)} \leq C' \vert\vert \phi\vert\vert_{L^2_{N} (\pi^{-1}(L)\cap \gamma F, \tS)} 
$$
The proof is now complete.
\end{proof}

\medskip

\begin{proposition}\label{Bounded}
For any finite subset $\maF:=\{\tm_j, 1\leq j\leq d\}$ of  the fundamental domain $F$, one defines a  bounded $\Gamma$-equivariant operator 
$$
T_\maF:\Ker (\tD) \longrightarrow \bigoplus_{j=1}^d \ell^2 (\Gamma \tm_j)\otimes S_{m_j} \text{ by setting }T_\maF \phi := (T(g \tm_j))_{\gamma\in \Gamma, 1\leq j\leq d}.
$$ 
\end{proposition}

\begin{proof}
Choose any $\alpha > 0$  such that  the balls $B(\tm_j, \alpha)$ centered at $\tm_i$ with radius $\alpha$ are disjoint from each other and contained in the open fundamental domain $F$. Then since the metric is $\Gamma$-invariant, for any $\gamma\in \Gamma$, the same property is satisfied by the collection $B(\gamma\tm_j, \alpha)$ inside $\gamma F$. Applying Proposition \ref{C1estimate} with any open relatively compact subspace $L$ of $M$ which contains the projection of the union of the balls $B(\tm_j, \alpha)$ for $1\leq j\leq d$, we deduce that for any $\tm\in B(\gamma\tm_j, \alpha)$ and for any $\phi\in \Ker (\tD)$, we have
$$
\vert \phi (\gamma\tm_j)\vert^2 \leq 2 \left(\vert \phi (\tm)\vert^2 + C^2 \alpha^2 \beta_\gamma(\phi)\right),
$$
where $C= C(L)$ is given by Proposition \ref{C1estimate}. 
Integrating this inequality over $B(\gamma\tm_j, \alpha)$ and suming  over $j\in \{1, \cdots, d\}$ and then over $\gamma\in \Gamma$, we deduce
$$
 \sum_{j=1}^d \frac{vol (B(\tm_j, \alpha))}{2} \sum_{\gamma\in \Gamma} \vert \phi (\gamma\tm_j)\vert^2 \leq \sum_{j, \gamma} \int_{B(\gamma\tm_j, \alpha)} \vert \phi (\tm)\vert^2 d\tm +  C^2 \alpha^2\sum_{j=1}^d  vol (B(\tm_j, \alpha)) \sum_\gamma \beta_\gamma (\phi).
$$
But
$$
\sum_\gamma \beta_\gamma (\phi) = \vert\vert \phi\vert\vert^2_{L^2 (\tM, \tS)}  \text{ and } \sum_{j, \gamma} \int_{B(\gamma\tm_j, \alpha)} \vert \phi (\tm)\vert^2 d\tm \leq \vert\vert \phi\vert\vert^2_{L^2 (\tM, \tS)}.
$$
Hence if $m=\inf_{1\leq j\leq d} vol (B(\tm_j, \alpha))$, and $M=\sup_{1\leq j\leq d} vol (B(\tm_j, \alpha))$, then we get
$$
\sum_{j, \gamma} \vert \phi (\gamma\tm_j)\vert^2 \leq  \frac{2 + 2 d C^2 \alpha^2 M}{m} \times \vert\vert \phi\vert\vert^2_{L^2 (\tM, \tS)}.
$$
Therefore,  the sum $\sum_{j, \gamma} \vert \phi (\gamma\tm_j)\vert^2$ converges so that 
$$
(\phi(\gamma\tm_j))_{j, \gamma}\in \bigoplus_{i=1}^d \ell^2(\Gamma \tm_j)\otimes S_{m_j}\simeq \ell^2\Gamma\otimes \C^{d\times\dim (S)}.
$$ 
Moreover, since the constant $ \frac{2 + 2 d C^2 \alpha^2 M}{m} $ does not depend on $\phi\in \Ker (\tD)$, the map 
$$
T_\maF: \Ker (\tD) \longrightarrow \bigoplus_{i=1}^d \ell^2(\Gamma \tm_j)\otimes S_{m_j} \text{ given by } T_\maF(\phi):= (\phi (\gamma\tm_j))_{j, \gamma} \text{ is a bounded operator}.
$$
It is finally obvious from its very definition that the operator $T_\maF$ is $\Gamma$-equivariant between the two Hilbert $\Gamma$-spaces $\Ker (\tD)$ and $\bigoplus_{i=1}^d \ell^2(\Gamma \tm_j)\otimes S_{m_j}$ where this latter is induced from the  regular  representation through the isomorphism $\ell^2 (\Gamma)\simeq \ell^2(\Gamma \tm_j)$.
\end{proof}

We are now in position to prove our theorem.

\begin{proof} (of Theorem \ref{Main})

Take any positive number $\epsilon\in ]0, 1]$ and fix a finite subset $\maF (\epsilon)=\{\tm_j, 1\leq j\leq d\}$ of $F\cap \pi^{-1} (K)$ such that the collection of open balls $(B(\tm_j, \epsilon))_{1\leq j\leq d}$ provides an open cover of $F\cap \pi^{-1} (K)$. Applying Proposition \ref{Bounded}, we know that the operator $T_{\maF (\epsilon)}$ from $\Ker (\tD)$ to the Hilbert $\Gamma$-space $\bigoplus_{j=1}^d \ell^2(\Gamma\tm_j)\otimes S_{m_j}$ given by evaluation at the points $\gamma\tm_j$ for $\gamma\in \Gamma$ and $1\leq j\leq d$, is a bounded $\Gamma$-equivariant operator. 
Assume that $\dim_\Gamma (\Ker (\tD)) = +\infty$, then we claim that the operator $T_{\maF(\epsilon)}$ cannot be injective. Indeed, if it were injective, then denoting by $H$ the Hilbert $\Gamma$-subspace of $\bigoplus_{j=1}^d \ell^2(\Gamma\tm_j)\otimes S_{m_j}$ which is the closure of the range of $T_{\maF(\epsilon)}$, we would get an injective $\Gamma$-equivariant operator from $\Ker (\tD)$ to $H$ which has dense range. But  this shows that $\dim_\Gamma (H)=+\infty$, by a standard argument using the normality of the trace $\tau$, see Lemma 4 in \cite{BenameurFack} and its proof. 
Hence, we can find $\phi\in \Ker (\tD)$ such that
$$
\vert\vert \phi\vert\vert_{L^2(\tM, \tS)} =1 \text{ and } \phi (\gamma\tm_j) = 0\text{ for any } \gamma\in \Gamma\text{ and }j\in \{1, \cdots, d\}.
$$
There then exists   $\gamma\in \Gamma$ such that $\beta_\gamma(\phi) = \int_{\gamma F} \vert \phi (\tm)\vert^2 d\tm >0$. Again replacing $\phi$ by $\frac{\phi}{\sqrt{\beta_\gamma(\phi)}}$ we can assume that $\beta_\gamma (\phi)=1$ and that $\phi$ vanishes on $\Gamma\maF(\ep)$. Denote now by $L$ any open relatively compact subspace of $M$ which contains $K$, one can take for instance the $1$-neighborhood  of $K$, say 
$$
L=\{m\in M, d(m, K)<1\}.
$$
Applying Proposition \ref{C1estimate},  we deduce  the existence of $C>0$, independent of $\epsilon$ and $\phi$, such that:
$$
\vert \phi (\tm) \vert \leq C \times \epsilon, \text{ for any }\tm\in \gamma F\cap \pi^{-1} (K).
$$
Hence $\vert \vert \phi\vert\vert_{\gamma F\cap \pi^{-1} (K)} \leq C'  \times \epsilon$ for some constant  $C'>0$ which is independent of $\epsilon$ and $\phi$. But Lemma \ref{PSCestimate} then allows to  conclude  that there exists a constant $C'>0$ such that 
$$
C'\epsilon \geq \frac{\kappa_0}{\kappa_0+c}.
$$
If $\epsilon$ is small enough,  we get a contradiction. Whence we conclude that $\dim_\Gamma (\Ker (\tD)) < +\infty$. 

Inorder to complete the proof of Theorem \ref{Main}, we now  prove with a different direct method  that $P_\ep$ has $\tau$-finite range. Notice that if $\sigma$ is a non trivial $L^2$-section which belongs to the range of $P_\ep$, then 
$$
 \vert\vert \tD\sigma\vert\vert^2\leq \ep \vert\vert \sigma\vert\vert ^2.
$$
Therefore and using our assumption that $\maR\geq \kappa_0 \Id$ off the compact subspace $K$, and that $\maR \geq -c\Id$ over the whole of $M$, we deduce that for $\ep$ small enough
$$
\vert\vert \sigma\vert\vert^2_{\pi^{-1}(K)} \geq \frac{\kappa_0-\ep}{\kappa_0+c} \vert\vert \sigma\vert\vert^2.
$$
Said differently, if we denote by $p_K$ the $\Gamma$-invariant orthogonal projection from $L^2(\tM, \tS)$ onto  the Hilbert subspace $L^2(\pi^{-1}(K), \tS)$, then for $\ep$ small enough, the restriction of $p_K$ to the range of $P_\ep$ is bounded below.
We now choose a bounded $\Gamma$-invariant operator (a parametrix) $\tQ$ such that $\tS=\Id - \tQ\tD$ is a smoothing operator with finite propagation. This can be achieved for instance by lifting to $\tM$ a localized enough near the diagonal pseudodifferential parametrix  for $D$ modulo smoothing opertors on $M$, see for instance \cite{AtiyahCovering}. Composing on the left with the projection $p_K$ we get
$$
p_K \tS = p_K - (p_K \tQ)\tD.
$$
Hence for $\sigma$ in the range of $P_\ep$, we can write
$$
\vert\vert p_K\tS \sigma\vert\vert \geq \vert\vert \sigma\vert\vert_{\pi^{-1} (K)} - \sqrt{\ep}\, \vert\vert \tQ\vert\vert \times \vert\vert \sigma\vert\vert\geq \left( \sqrt{\frac{\kappa_0-\ep}{\kappa_0+c}} - \sqrt{\ep}\, \vert\vert \tQ\vert\vert\right) \times \vert\vert \sigma\vert\vert.
$$
Hence for small enough $\ep$, the operator $p_K\tS$ is a $\Gamma$-equivariant  isomorphism from the range of $P_\ep$ to the range of $p_K\tS P_\ep$, the two being closed $\Gamma$-invariant subspaces of $L^2(\tM, \tS)$. Since $p_K\tS$ is a compact operator relative to the von Neumann algebra $\maM$, we conclude that  $P_\ep$ must be a compact operator relative to $\maM$ and therefore has a finite $\tau$-trace. Indeed, one easily checks for instance that $p_K\tS$ is a Hilbert-Schmidt operator relative to the trace $\tau$, see again \cite{AtiyahCovering}.

\end{proof}

\medskip

\section{Compatibility  with the  higher index}

\subsection{Review of the Xie-Yu higher index}

Given two Hilbert modules $\maE_1$ and $\maE_2$, the space of adjointable operators from $\maE_1$ to $\maE_2$ will be denoted $\Mor (\maE_1, \maE_1)$ and when $\maE_1=\maE_2=\maE$ then we denote the resulting unital $C^*$-algebra by $\Mor (\maE)$. Recall that the subspace $\maK(\maE_r)$ of $\Mor (\maE_r)$ composed of compact operators is a closed two-sided involutive ideal, see \cite{KasparovStinespring} for more details. In \cite{XieYu}, Xie and Yu proved that when the spin even dimensional manifold $M$ has a complete metric $g$ with PSC near infinity, then the operator $\maD_r$ obtained by twisting the spin-Dirac operator $D_g$ with the flat reduced Michschenko bundle, admits a well defined higher index class 
$$
\Ind (\maD_r) \in K_0 (C_r^*\Gamma).
$$
Here $C^*_r\Gamma$ is  the regular $C^*$-algebra associated with $\Gamma$ and $\maD_r$ acts on the $L^2$ sections of $S\otimes \Xi$, with $\Xi = \tM\times_\Gamma C^*_r\Gamma$ being the flat Michschenko  bundle with fibers $C^*_r\Gamma$ viewed as a module over itself.  It is well known that $\maD_r$ is essentially self-adjoint and yields a regular self-adjoint operator, still denoted $\maD_r$, with domain contained (in general strictly) in the Hilbert module Sobolev space $L^2_1 (M, S\otimes \Xi)$ that one defines using the spin connection $\nabla$ as before  twisted by the flat connection on the bundle of modules $\Xi$. 

Let us recall the construction of the higher index class from \cite{XieYu}, which as we shall see  remains valid for all generalized Dirac operators which are invertible near infinity. We thus assume that the generalized Dirac operator $D$ on $M$ has square given by $\nabla^*\nabla+\maR$ with a zero-th order operator $\maR$ such that  $\maR\geq  \kappa_0\Id$ off some compact subspace $K$ of $M$. There exists then a smooth compactly supported real valued function $\rho\in C_c^\infty (M)$ and a constant $c>0$ such that the pointwise operator inequality 
$$
 \maR \geq (c - \rho^2) \Id,
$$ 
holds over the whole of $M$. This discussions works as well for the Hilbert module operator $\maD_r$ by replacing $\maR$ by its tensor product still denoted $\maR$ by the fiberwise identity of $\Xi$.
We shall denote as well by $\rho$ the smooth function on $\tM$ which is the pull-back of $\rho$.
We concentrate on the construction with the regular completions (the maximal completions are similar) and we  add some details for the convenience of the reader.  Let us consider the operator $\maF:=\maD_r(\maD_r^2+\rho^2)^{-1/2} $ acting on  $\maE_r$, where $(\maD_r^2+\rho^2)^{-1/2}$ is by definition the square root of the self-adjoint non-negative \underline{bounded} operator $(\maD_r^2+\rho^2)^{-1}$, see \cite{XieYu}. As we shall see this operator can also be defined equivalently as the inverse of the square root of the regular self-adjoint non-negative operator $\maD_r^2+\rho^2$. Another equivalent and convenient definition is
$$
(\maD_r^2+\rho^2)^{-1/2} := \frac{1}{\pi} \int_0^{+\infty}\;  \frac{ (\maD_r^2+\rho^2)^{-1}}{ (\maD_r^2+\rho^2)^{-1} + \mu}\;\;  \frac{d\mu}{\sqrt{\mu}},
$$
where the integral converges in the operator norm.  The operator $\maF$ is in fact an adjointable operator as we shall see. Notice first that we take a priori as initial  domain 
$$
\Dom (\maF)=\{\sigma\in \maE_r\, \vert\, (\maD_r^2+\rho^2)^{-1/2}\sigma\in \Dom (\maD_r)\},
$$
and $\maF$ is then closed.
In the same way the operator $\maG:= (\maD_r^2+\rho^2)^{-1/2} \maD_r$ is closed  when we take as domain that of the operator $\maD_r$. Moreover,   for $\sigma\in \Dom (\maF)$ and $\sigma'\in \Dom (\maD_r)$, we have
$$
\langle \maF\sigma, \sigma'\rangle = \langle \sigma, \maG\sigma'\rangle,
$$
and the domain of $\maF$ can be seen to be exactly equal to the adjoint domain of  $\maG$.

\begin{lemma}
The operator $\maG\maF$  is a bounded (self-adjoint and non-negative) operator on the Hilbert module $\maE_r$. In particular, so is $\maF^*\maF$ and we have $\maF^*\maF=\maG\maF$. 
\end{lemma}

\begin{proof}
Notice that $
\maG\maF = (\maD_r^2+\rho^2)^{-1/2}\maD_r^2  (\maD_r^2+\rho^2)^{-1/2}$
has dense domain  in $\maE_r$.  It is easy to check that the operator $ (\maD_r^2+\rho^2)^{-1/2}(\maD_r^2+\rho^2)  (\maD_r^2+\rho^2)^{-1/2}$ extends to the identity operator so that we can write for $\sigma \in \Dom (\maG\maF)$:
$$
\maG\maF \sigma = \sigma -  (\maD_r^2+\rho^2)^{-1/2}\rho^2  (\maD_r^2+\rho^2)^{-1/2} \sigma.
$$
The operator $(\maD_r^2+\rho^2)^{-1/2}\rho^2  (\maD_r^2+\rho^2)^{-1/2}$ being bounded, we conclude that the operator $\maG\maF$ extends to a bounded operator on $\maE_r$ which coincides with the operator $\Id - (\maD_r^2+\rho^2)^{-1/2}\rho^2  (\maD_r^2+\rho^2)^{-1/2}$. Moreover, if $\sigma'\in \Dom (\maF)$ is such that $\maF\sigma'\in \Dom (\maG)=\Dom(\maD_r)\subset \Dom (\maF^*)$ then $\maF\sigma'\in \Dom (\maF^*)$ and we can write for any $\sigma\in \maE_r$
$$
\langle \maG\maF\sigma, \sigma'\rangle = \langle \maF\sigma, \maF\sigma'\rangle = \langle \sigma, \maF^*\maF\sigma'\rangle = \langle \sigma, \maG\maF\sigma'\rangle.
$$
We have used the relations $\maG\subset \maF^*$ and $\maF=\maG^*$. Notice finally that for any $\sigma\in \Dom (\maD_r^2)$ we have
$$
\langle \rho^2\sigma, \sigma\rangle \leq \langle (\maD_r^2+\rho^2)\sigma, \sigma\rangle.
$$
Therefore the operator $\maG\maF$ is indeed non-negative. Recall that $\maG\maF\subset \maF^*\maF$, hence we conclude that we also the equality $\maG\maF=\maF^*\maF$ of adjointable operators.
\end{proof}

\begin{remark}
Using the above integral expression, it is easy to check that the densely defined commutator $[\maD_r, (\maD_r^2+\rho^2)^{-1/2}]$ extends to a compact adjointable operator. This is a consequence of the fact that $\rho$ and $[\maD_r, \rho]$ have compact supports in $M$ and $(\maD_r^2+\rho^2)^{-1/2}$ is bounded from $L^2$ to $L^2_1$, so that an easy application of a Rellich lemma in this context, see \cite{XieYu} allows to conclude. Hence  the operator $\maF - \maG$ is a compact operator on the Hilbert module $\maE_r$. 
\end{remark}

\medskip

\begin{proposition}\cite{XieYu}
The closed operators $\maF$ and $\maG$ are adjointable. Moreover, the operators $\Id-\maF^2$ and $\Id - \maG^2$ are compact operators of the Hilbert module $\maE_r$. 
\end{proposition}

\begin{proof}
Let us check first that the closed operator $\maF$ is bounded.  By the previous lemma, we know that if $\maF^*$ is the adjoint of the closed operator $\maF$ then the operator $
\maF^*\maF$ is a bounded non-negative operator on the Hilbert module $\maE_r$. Therefore the operator $\Id +\maF^*\maF$  is an adjointable self-adjoint  invertible operator, and hence is surjective so that $\maF$ is in fact a regular  operator on $\maE_r$.  The square root of $\Id +\maF^*\maF$ is then also an adjointable self-adjoint invertible operator with the inverse given precisely by our previously defined operator $(\Id + \maF^*\maF)^{-1/2}$, by the classical properties of continuous functional calculus of regular self-adjoint operators. Now, $\maF$ being regular, by the properties of the Woronowicz transform \cite{SkandalisCourse}, we  see that the operator $\maF (\Id + \maF^*\maF)^{-1/2}$ is an adjointable operator. We conclude that the operator 
$$
\maF = \maF (\Id + \maF^*\maF)^{-1/2} (\Id + \maF^*\maF)^{1/2},
$$
is adjointable with the adjoint given by $\maF^*$. As a corollary, we deduce that the densely defined operator $\maG$ is also extendable to an adjointable operator which coincides with $\maF^*$. 

Computing on a dense subspace, one deduces the following equality of adjointable operators
\begin{eqnarray*}
\Id - \maF^2 = (\maD_r^2+\rho^2)^{-1/2}\rho^2 (\maD_r^2+\rho^2)^{-1/2} + [(\maD_r^2+\rho^2)^{-1}, \maD_r] \maF.
\end{eqnarray*} 
As we have already pointed out, the operator $[(\maD_r^2+\rho^2)^{-1}, \maD_r]$ is a compact operator on the Hilbert module $\maE_r$ by a generalized version of the Rellich lemma \cite{XieYu}, and $\maF$ is adjointable so that the second term in the RHS is compact. On the other hand, and also by the compact support of $\rho^2$, the fact that  $(\maD_r^2+\rho^2)^{-1/2}$ is bounded with image in $L^2_1$, we deduce again by the Rellich lemma in the Mischenko-Fomenko calculus \cite{XieYu} that the operator $(\maD_r^2+\rho^2)^{-1/2}\rho^2 (\maD_r^2+\rho^2)^{-1/2}$ is a compact operator on the Hilbert module $\maE_r$. This ends the proof for $\maF$. Since $\maG=\maF^*$, the proof of the proposition is now complete.
\end{proof}

We deduce from the previous lemma that the operator 
$$
\maG^+:= (\maD_r^+\maD_r^- +\rho^2)^{-1/2} \maD_r^+: \maE_r^+\longrightarrow \maE_r^-,
$$
is a Fredholm operator between these Hilbert modules with quasi inverse given by 
$$
\maG^-:=  (\maD_r^-\maD_r^+ +\rho^2)^{-1/2}\maD_r^-: \maE_r^-\longrightarrow \maE_r^+.
$$
Therefore $\maG^+$ admits a well defined index class in $K_0 (\maK (\maE_r))$ where $\maK (\maE_r)$ is the $C^*$-algebra of compact operators on the Hilbert module $\maE_r$. We may as well use $\maF^+$ and we would end up with the same class as already observed. This index class  is given more precisely by the $K$-class $[e]-[f]$ where $e$ and $f$ are the two self-adjoint projections on $\maE_r$ which satisfy that $e-f\in \maK (\maE_r)$:
$$
e:=\left(\begin{array}{cc} (\maS^+)^2 & \maS^+ \maG^- \\ \maG^+\maS^+ (\Id_{\maE_r^+}+\maS^+) & \Id_{\maE_r^-} -(\maS^-)^2 \end{array}\right) \text{ and } f=\left(\begin{array}{cc} 0  & 0 \\ 0 & \Id_{\maE_r^-} \end{array}\right)
$$
where 
$$
\maS^+ = \Id_{\maE_r^+} - \maG^-\maG^+\text{ and }\maS^- = \Id_{\maE_r^-} - \maG^+ \maG^-.
$$
That $[e]-[f]$ defines a class in $K_0(\maK (\maE_r))$ is standard.

 \medskip

\begin{definition}\cite{XieYu}\
The higher index class $\Ind (\maD_r^+)$ of $\maD_r^+$ is by definition the image of the index class of $\maG^+$ in $K_0(C^*_r\Gamma)$ under the isomorphism $K_0 (\maK (\maE_r))\simeq K_0(C^*_r\Gamma)$ induced by the Morita equivalence $\maK (\maE_r)\sim C^*_r\Gamma$, so $
\Ind (\maD_r^+) \in K_0(C^*_r\Gamma).$
\end{definition}

\medskip

\begin{remark}
From the relations $\Id - \maF^2\in \maK(\maE_r)$ and $\Id - \maG^2\in \maK(\maE_r)$, we deduce that the operator $\maQ:=(\maD^-_r\maD_r^+ +\rho^2)^{-1/2} \maD^-_r  (\maD_r^+\maD_r^- +\rho^2)^{-1/2}$ is a parametrix for $\maD^+_r$ with remainders in $\maK (\maE_r^\pm)$.
\end{remark}

We may as well give the same construction of the maximal index class $\Ind (\maD_m^+)$ living in $K_0(C^*_m\Gamma)$ by using maximal completions everywhere. Notice that  by using the natural $C^*$-homomorphism from $C^*_m\Gamma$ to  $C^*_r\Gamma$, one can as well recover the  index class $\Ind (\maD_r)$  as the image of $\Ind (\maD_m)$, see for instance \cite{BenameurPiazza}. 

\medskip

\subsection{Traces and numerical indices}

The trivial $1$-dimensional representation of $\Gamma$ gives a $C^*$-algebra homomorphism $C^*_m\Gamma \to \C$, this  is the so-called average trace $\tau^{\av}$. In the same way the regular trace $\tau^{\reg}$ is a finite trace on $C^*_r\Gamma$. These traces are defined on finitely supported functions $f\in \C\Gamma$ by
$$
\tau^{\av} (f) := \sum_{\gamma\in \Gamma} f(g)\text{ and } \tau^{\reg} (f) := f(e).
$$
They induce the group homomorphisms 
$$
\tau^{\av}_* : K_0(C^*_m\Gamma) \longrightarrow \R\quad \text{ and } \quad \tau^{\reg}_* : K_0(C^*_r\Gamma) \longrightarrow \R.
$$ 
\vspace{0,01cm}
\begin{theorem}\label{Compatibility}
The following relations hold
$$
\tau^{\av}_* (\Ind (\maD_m)) = \Ind (D)\; \text{ and } \; \tau^{\reg}_* (\Ind (\maD_r)) = \Ind_{(2)} (\tD),
$$
where $\Ind (D)$ is the Gromov-Lawson index of  the generalize Dirac operator $D$ on $M$ while  $\Ind_{(2)} (\tD)$ is our $L^2$ index of the lifted  $\Gamma$-invariant generalized Dirac operator $\tD$ on $\tM$.
\end{theorem}

\medskip
\begin{remark}
We shall fully use for the proof of the second relation in Theorem \ref{Compatibility} our Theorem \ref{Main}, say that there exists $\ep>0$ such that $P_\ep$ is $\tau$-trace class in the von Neumann algebra $\maM$.
\end{remark}

\begin{proof}
Denote by $\pi^{\reg}$ and $\pi^{\av}$ the regular and average representations of $C^*_r\Gamma$ and $C^*_m\Gamma$ respectively that we can see as being both representations of $C^*_m\Gamma$ by using the natural morphism $C^*_m\Gamma\to C^*_r\Gamma$. By \cite{BenameurRoy}[Propositions 2.3 $\&$ 2.4], the composition Hilbert spaces $\maE_r\otimes_{\pi^{\reg}} \ell^2\Gamma$ and $\maE_m\otimes_{\pi^{\av}} \C$ are respectively isomorphic to $L^2(\tM, \tS)$ and $L^2(M, S)$ through the explicit isomorphisms denoted there $\Psi_{\reg}$ and $\Psi_{\av}$. Let us concentrate on the second equality first. The regular self-adjoint operator $\maD_r$ gives after composition with the regular representation and through conjugation with $\Psi_{\reg}$ the self-adjoint operator $\tD$ \cite{BenameurRoy}, i.e.
$$
\Psi_{\reg} (\maD_r\otimes_{\pi^{\reg}} \Id)\Psi_{\reg}^{-1} = \tD.
$$
The mapping $\Phi_{\reg}: T\mapsto \Psi_{\reg} (T\otimes_{\pi^{\reg}} \Id)\Psi_{\reg}^{-1}$ induces a well defined $C^*$-algebra morphism between the adjointable operators on $\maE_r$ and the bounded operators on $L^2(\tM, \tS)$ which  belong to the von Neumann algebra $\maM$ \cite{BenameurRoy}.
The last property that we shall need from \cite{BenameurRoy} is that $\Phi_{\reg}$ sends the ideal $\maK(\maE_r)$ of compact operators to the ideal of $\tau$-compact operators in $\maM$ and intertwines the corresponding short Calkin exact sequences, i.e.
\vspace{0,2cm}
\begin{equation*}\label{beta}
\hspace{-0,47cm}
\begin{CD}
0\to \maK(\maE_r) @>>>  \Mor (\maE_r)  @>>>  \Mor (\maE_r)/\maK(\maE_r) \to 0 \\
 \;\;@  V{\Phi{\reg}\;}VV                          @V{\Phi_{\reg}}VV       @V{\Phi_{\reg}}VV                           \\
 0\to  \maK (\maM, \tau) @>>>  \maM @>>> \maM/ \maK (\maM, \tau)\to 0  \\
\end{CD}
\end{equation*}
\vspace{0,2cm}
As a corollary, the boundary maps in $K$-theory are compatible, say
$$
\Phi_{\reg, *} \circ \pa = \pa \circ \Phi_{\reg, *}  \; :\; K_1 \left(\Mor (\maE_r)/\maK(\maE_r)\right)\longrightarrow K_0 \left(\maK (\maM, \tau)\right).
$$
We then immediately deduce the following relation
$$
\Phi_{\reg, *} \left[\Ind (\maG^+)\right] = \Ind_\maM \left( \Phi_{\reg} (\maG^+)\right),
$$
where $\Ind_\maM : K_1 (\maM/\maK (\maM, \tau))\rightarrow K_0(\maK (\maM, \tau))$ is the boundary map associated with the second exact sequence, and constructed as usual using any parametrix in $\maM$ modulo $\maK(\maM, \tau)$. 
Now, recall that $\Ind (\maD_r^+)=\Ind (\maG^+)$ by definition, and by compatibility of functional calculi with respect to compositions of Hilbert modules \cite{SkandalisCourse}, we have
$$
\Phi_{\reg} (\maG^+) = (\tD^+\tD^- +\rho^2)^{-1/2} \tD^+=:\tG^+: L^2(\tM, \tS^+) \longrightarrow L^2 (\tM, \tS^-).
$$
The operator $(\tD^+\tD^- +\rho^2)^{-1/2}$ is thus defined by the continuous functional calculus for bounded $\Gamma$-invariant operators on $L^2(\tM, \tS)$ from the bounded operator $(\tD^+\tD^- +\widetilde{\rho}^2)^{-1}$. 

Now we can use any parametrix  for $\tG^+$ in $\maM$ (not necessarily in the range of $\Phi_{\reg}$)  inorder to represent the index class $\Ind (\tG^+)\in K_0(\maK (\maM, \tau))$. Moreover, since $\tD$ is affiliated with $\maM$, we can also use directly any parametrix for the operator $\tD$. Now, by Theorem \ref{Main}, there exists $\ep>0$ such that the spectral projection $P_\ep:=1_{[0, \ep]} (\tD^2)$ is $\tau$-trace class. We consider for such $\ep$ the odd bounded Borel function $f_\ep : \R\to \R$ given by
$$
f_\ep (x):=\left\{\begin{array}{ccc}\frac{1}{x} &\text{ if } & \vert x\vert > \sqrt{\ep}\\
0 &\text{ if } & \vert x\vert \leq \sqrt{\ep} \end{array}\right.
$$
Then set $\tQ_\ep:= f_\ep (\tD)$, an element of $\maM$ which is odd for the grading. Since $1-xf_\ep (x) = 1_{[0, \ep]} (x^2)$, we have
$$
\Id - \tQ_\ep \tD = \Id - \tD \tQ_\ep = P_\ep \text{ and }\tQ_\ep = \left(\begin{array}{cc}0 & \tQ_\ep^-\\ \tQ_\ep^+ & 0\end{array}\right).
$$
Moreover, the operator $\tH := \tQ_\ep (\tD^2+\widetilde{\rho}^2)^{1/2}$ is also a bounded operator which belongs to $\maM$ and which is then a parametrix for the operator $\tG$. Indeed, one has for any $C>0$:
$$
 \tQ_\ep (\tD^2+C^2)^{1/2} = g_\ep (\tD) \text{ with } g_\ep (x) = \left\{\begin{array}{ccc} x/\sqrt{x^2+C^2}& \text{ if } \vert x\vert >\sqrt{\ep}\\ 0 & \text{ if } \vert x\vert\leq \sqrt{\ep}\end{array}\right.
$$
Therefore,  $\tQ_\ep (\tD^2+C^2)^{1/2}$ extends to a bounded operator which lives in $\maM$. Moreover, the operator $(\tD^2+C^2)^{-1/2}(\tD^2+\rho^2)^{1/2}$ also extends to a bounded operator which belongs to $\maM$. Hence, we deduce that  $\tH$ belongs to $\maM$. On the other hand, we have
$$
\tH\tG =  \Id - P_\ep \text{ while } \tG\tH = \Id - (\tD^2+\rho^2)^{-1/2} P_\ep \rho^2 (\tD^2+\rho^2)^{-1/2}.
$$
Therefore $\tH$ is a parametrix for $\tG$ modulo $\tau$-trace class operators in $\maK (\maM, \tau)$. Therefore, the index class of $\tG$ can be represented by the two idempotents $e$ and $f$ in $\maM$ such that $e-f\in \maK (\maM, \tau)$  given by
$$
\left(\begin{array}{cc} P_\ep^+ &    0 \\
2 (\tD^-\tD^++\rho^2)^{-1/2} \tD^+P_\ep^+ & - (\tD^+\tD^- +\rho^2)^{-1/2} P_\ep^-  (\tD^+\tD^- +\rho^2)^{1/2} \end{array}\right)
$$
with $f$ as before given by $f=\left(\begin{array}{cc} 0 & 0 \\ 0 & \Id_{L^2(\tM, \tS^-)}\end{array}\right)$. Notice that the operator $P_\ep^-  (\tD^+\tD^- +\widetilde{\rho}^2)^{1/2}$ is bounded and even $\tau$-trace class since it equals
$$
(P_\ep^- \tD^+\tD^- + P_\ep \rho^2) (\tD^+\tD^- +\rho^2)^{-1/2},
$$
and the operator $P_\ep^- \tD^+\tD^-= P_\ep (P_\ep^- \tD^+\tD^-)$ is bounded and even $\tau$-trace class, and finally the operator $(\tD^+\tD^- +\rho^2)^{-1/2}$ is bounded.

If we now apply the trace $\tau$ to this $K$-theory class, then we get
$$
(\tau_{*} \circ \Ind_\maM) (\tG^+) = \tau (P_\ep^+) - \tau ((\tD^+\tD^- +\rho^2)^{-1/2} P_\ep^-  (\tD^+\tD^- +\rho^2)^{1/2})
$$
Since $\tau$ is a hypertrace and again since $P_\ep^-  (\tD^+\tD^- +\widetilde{\rho}^2)^{1/2}$ is $\tau$-trace class while $(\tD^+\tD^- +\widetilde{\rho}^2)^{-1/2}$ is bounded and belongs to $\maM$, we can write
$$
\tau ((\tD^+\tD^- +\rho^2)^{-1/2} P_\ep^-  (\tD^+\tD^- +\rho^2)^{1/2}) = \tau ( P_\ep^-  (\tD^+\tD^- +\rho^2)^{1/2}(\tD^+\tD^- +\rho^2)^{-1/2})= \tau (P_\ep^-).
$$
We conclude that 
$$
(\tau_{*} \circ \Ind_\maM) (\tG^+) = \tau (P_\ep^+) - \tau (P_\ep^-)\text{ for any small enough }\ep>0.
$$
Since the trace $\tau$ is normal, we can deduce from this equality that
$$
(\tau_{*} \circ \Ind_\maM) (\tG^+) = \tau (P_0^+) - \tau (P_0^-) = \Ind_{(2)} (\tD^+).
$$
We conclude from the whole previous discussion that 
$$
 \Ind_{(2)} (\tD^+) = \left[(\tau_{*}\circ \Phi_{\reg, *}) \circ \Ind)\right] (\maD_r^+) = \tau^{\reg}_*(\Ind (\maD_r^*))
$$
The last equality is the consequence of the compatibility of the traces $\tau$ and $\tau^{\reg}$ with the Morita equivalence $\maK(\maE_r)\sim C^*_r\Gamma$, say
$$
(\tau_{*}\circ \Phi_{\reg, *}) =\tau^{\reg}_*\circ M: K_0(\maK(\maM, \tau) \longrightarrow \R,
$$
with $M:K_0(\maK(\maE_r))\to K_0(C^*_r\Gamma)$ being the  isomorphism induced by the Morita equivalence.  Details of the  proof of this latter result can also be found in \cite{BenameurRoy}.

The proof for the trivial representation and the average trace is similar and actually simpler, since the von Neumann algebra $\maM$ has to be replaced by the bounded operators on the Hilbert space $L^2(M, S)$ and the spectrum of $D$ admits a gap around zero, which allows to use a parametric for $D$ which is the bounded Green operator. We leave these straightforward verifications for the interested reader, who can alternatively also  use  the needed compatibility results proved in \cite{BenameurRoy}.
\end{proof}

\section{Relative $L^2$ index theory}

\subsection{The $\Phi$-relative $L^2$ index theorem}

The relative index theorem of Gromov-Lawson \cite{GromovLawson}[Theorem 4.18] can clearly be extended to Galois coverings with the expected statement about pairs of  Galois coverings $\pi: \tM\to M$ and $\pi': \tM'\to M'$ and $\Gamma$-operators $\tD$ and $\tDp$ which are invertible near infinity. For simplicity we restrict ourselves to the case of the spin Dirac operators for complete metrics which have PSC near infinity. More precisely, we assume that $M$ and $M'$ are endowed with complete metrics $g$ and $g'$ which both have uniform  PSC outside compact subspaces $K$ and $K'$ in $M$ and $M'$ respectively, and associated spin structures.  Then the $\Gamma$-invariant Dirac operators $\tD=\tD_g$ and $\tDp=\tDp_{g'}$ acting on  the sections of the hermitian spinor bundles $\tS$ and $\tS'$ both have well defined $L^2$ indices $\Ind_{(2)} (\tD)$ and $\Ind_{(2)} (\tDp)$, thanks to Theorem \ref{Main}.

Moreover, we assume as in \cite{GromovLawson} that there exists  open submanifolds $\Phi\in M\smallsetminus K$ and $\Phi'\subset M'\smallsetminus K'$ which are unions of connected components of $M\smallsetminus K$ and $M'\smallsetminus K'$ respectively, such that all structures over $\Phi$ and $\Phi'$ are identified. So denoting by $\tPhi:=\pi^{-1} (\Phi)$ and $\tPhi':={\pi'}^{-1} (\tPhi')$, we assume that there exists a $\Gamma$-equivariant smooth spin-preserving isometry
$$
\tPhi \rightarrow \tPhi'\text{ which is then covered by a bundle isometry } : S\vert_{\Phi} \rightarrow  S'\vert_{\Phi'},
$$
 inducing a unitary $U: L^2(\tPhi, \tS) \rightarrow L^2(\tPhi', \tS')$ so that the Dirac operators are conjugated over $\tPhi$ and $\tPhi'$, that is as self-adjoint operators:
$$
\tDp\vert_{\Phi'} = U \circ \tD\vert_{\Phi}\circ U^{-1}.
$$

As in \cite{GromovLawson}[page 127], we consider a compact hypersurface $H$ which is contained in $\Phi\simeq \Phi'$ which separates off the infinite part of $\Phi\simeq \Phi'$. More precisely, we assume that  there exists a compact subspace $L$ of $M$ containing $H$  and $K$  such that $\Phi\smallsetminus L$ and $K$ are contained in different connected components of $M\smallsetminus H$. The similar assumption is thus insured as well in $M'$ when we view $H$ in $M'$.  We may assume, up to modifying the metrics and spin structures near $H$, that the metrics are product metrics  and the Dirac operators are of product type,  in a collar neighborhood $(-\ep, +\ep) \times H$ of $H$ in $\Phi\simeq \Phi'$. Like in \cite{GromovLawson}[page 127], we chop off $M$ and $M'$ along $H$ and attach the resulting riemannian manifolds with boundaries $M_1$ and $M'_1$ to build up the new complete riemannian spin manifold $M'':=M_1\amalg_H M_1'$ whose metric has as well PSC near infinity. The spin structure (and  orientation) on $M''$ agrees with that on $M_1$ and with the opposite one on $M'_1$ corresponding to the opposite Stieffel-Whitney $2$-class, and the Dirac operator on $M''$ for the resulting spin structure is denoted $D''$ and agrees with $D$ on $M_1$ and with $D^{op}$ on $M_1'$. 

The Galois $\Gamma$-coverings $\tM\to M$ and $\tM'\to M'$ then yield a Galois $\Gamma$-covering $\tM''\to M''$ in an obvious way. Indeed, $H$ can be lifted to $\tH=\pi^{-1}(H)$ and all the constructions can be achieved upstairs in $\tM$ and $\tM'$ to build up the complete spin covering $\tM''\to M''$.

\begin{definition}
The $\Phi$-relative $L^2$ index of $(M, M', \Phi)$ is the $L^2$-index of the spin-Dirac operator ${\widetilde{D''}}^+$, i.e.
$$
\Ind_{(2)} (\tD, \tDp; \Phi) := \Ind_{(2)} ({\widetilde{D''}}^+) \quad\in \R.
$$ 
\end{definition}

We are now in position to state the main theorem of this paragraph. 
\begin{theorem}[The $\Phi$-relative $L^2$-index theorem]\label{PhiRelative}
Under the previous notations, we have
$$
\Ind_{(2)} (\tD, \tDp; \Phi) = \Ind_{(2)} (\tD^+) - \Ind_{(2)} ({\tDp}^+)
$$
\end{theorem}

\begin{proof}
Although this theorem can be proved directly by extending the Gromov-Lawson method, we shall deduce it here from the main theorem of \cite{XieYu}[Theorem 4.2] by using our compatibility Theorem \ref{Compatibility}. Indeed, the construction of the higher $\Phi$-relative index $\Ind (\maD, \maD'; \Phi)\in K_0(C^*_r\Gamma)$ carried out in \cite{XieYu} shows that its image under the group homomorphism induced by the regular trace $\tau^{\reg}$ on $C^*_r\Gamma$ coincides with our $\Phi$-relative $L^2$-index  $\Ind_{(2)} (\tD, \tDp; \Phi)$ since this is a consequence of the compatibility result on Galois coverings again. The main theorem of \cite{XieYu} gives the relation
$$
\Ind (\maD, \maD'; \Phi) = \Ind (\maD^+) - \Ind ({\maD'}^+)
$$
Therefore, using Theorem \ref{Compatibility}, we conclude that
\begin{eqnarray*}
\Ind_{(2)} (\tD, \tDp; \Phi) & = & \tau^{\reg}_* \left(\Ind (\maD, \maD'; \Phi)\right) \\
& = & \tau^{\reg}_* \left(\Ind (\maD^+)\right) - \tau^{\reg}_* \left(\Ind ({\maD'}^+)\right)\\
&  = &\Ind_{(2)} (\tD^+) -  \Ind_{(2)} ({\tDp}^+).
\end{eqnarray*}
\end{proof}

When the open submanifolds $\Phi$ and $\Phi'$ have compact complements in $M$ and $M'$ respectively, the $\Phi$-relative $L^2$ index can easily be sees to be independant of the choices of $\Phi$ and $\Phi'$ and it will be denoted $\Ind_{(2)} (\tD, \tDp)$. Moreover, in this case where the operators are identified near infinity, no need of any invertibility condition near infinity and the relative $L^2$ index can be defined exactly using the process of chopping off and pasting exactly as  in \cite{GromovLawson}[page 119] that is now done with the Galois coverings. We can deduce now the following important consequence.

\medskip

\begin{theorem}[The relative $L^2$-index theorem]\label{relative}
Under the assumptions of Theorem \ref{PhiRelative} with $M\smallsetminus \Phi$ and $M'\smallsetminus \Phi'$ compact subspaces of $M$ and $M'$ respectively, but without any condition on $\maR$ now. The relative $L^2$-index $\Ind_{(2)} (\tD, \tDp)$ is well defined, and we have
$$
\Ind_{(2)} (\tD, \tDp)  = \Ind (D, D').
$$
In particular, in this case $\Ind_{(2)} (\tD, \tDp)\in \Z$. 
\end{theorem}
\medskip
Here $\Ind (D, D')$ is the Gromov-Lawson relative index for the Dirac operators downstairs in $M$ and $M'$. Recall as well that 
$$
\Ind (D, D') = \int_U \widehat{A} (TM) - \int_{U'} \widehat{A} (TM'),
$$
where $U$ and $U'$ are open relatively compact neighborhoods of $L$ and $L'$ respectively whose complements correspond under the isometry. This is indeed the content of the classical Gromov-Lawson relative index theorem, see  \cite{GromovLawson}[Theorem 4.18].

\begin{proof}
From their very definitions, the indices $\Ind_{(2)} (\tD, \tDp)$ and  $\Ind (D, D')$ coincide respectively with the $L^2$-index of the $\Gamma$-invariant generalized Dirac operator on $\tM''$ and with the classical index of the generalized Dirac operator on $M''$. But here $M''$ is a closed manifold so that we can apply the classical Atiyah $L^2$-index theorem for Galois coverings of closed manifolds and deduce the allowed equality
$$
\Ind_{(2)} (\tD, \tDp)  = \Ind (D, D').
$$
\end{proof}

Another  direct consequence of the relative $L^2$-index Theorem \ref{PhiRelative} when we don't assume anymore that $M\smallsetminus \Phi$  and $M'\smallsetminus \Phi'$ are compact, is the following convenient result.

\medskip

\begin{theorem}
Under the assumptions of Theorem \ref{PhiRelative}, assume that there exists a Galois $\Gamma$-covering $\pi_N:\tN\to N$ of spin manifolds with boundary such that $N$ is compact, $\pi_N^{-1} (\pa N) = \pa \tN$ and the resulting Galois $\Gamma$-covering $\pa\tN\to \pa N$ is isometric to the Galois $\Gamma$-covering $\tH\to H$.   We  form as above the spin complete manifolds $N_1:=M_1\amalg_H N$ and $N'_1:= M'_1\amalg_H N$ with the resulting Dirac operators $D_1$ and $D'_1$ and their lifts $\tD_1$ and $\tDp_1$ to the covers.  Then
$$
\Ind_{(2)} (\tD, \tDp; \Phi) = \Ind_{(2)} (\tD_1^+) - \Ind_{(2)} ({\tDp_1}^+).
$$
\end{theorem}
\medskip

\begin{proof}
Indeed, the relative $L^2$-index theorem applies to prove that the RHS coincides with the $L^2$-index of the Dirac operator on the smooth complete manifold obtained by chopping off $N_1$ and $N_2$ along $H$ and attaching $M_1$ to $M'_1$. But this latter complete spin manifold is exactly the manifold $N$ used above to define $\Ind_{(2)} (\tD, \tDp; \Phi)$.  Since all these construction can be done upstairs $\Gamma$-equivariantly, we immediately deduce from Theorem \ref{relative} the allowed equality.
\end{proof}

\begin{remark}
The $\Phi$-relative $L^2$ index Theorem \ref{PhiRelative}, as well as all its corollaries listed above, can be proved directly without using the higher version obtained in \cite{XieYu}, by adapting the proof given in \cite{GromovLawson} to the type II von Neumann setting. 
\end{remark}

\medskip

\subsection{The torsion-free case}

In this second paragraph of applications, we shall show that the {\em absolute} Gromov-Lawson index, in opposition to the relative one, already coincides with its $L^2$-version when the group $\Gamma$ satisfies some extra conditions. Recall the famous Baum-Connes map and Baum-Connes conjecture for the group $\Gamma$ \cite{BaumConnes}. In the torsion-free case, the conjecture concerns the bijectivity of the assembly map
$$
\mu_\Gamma: K_*(B\Gamma) \longrightarrow K_*(C^*_r\Gamma).
$$
This map $\mu_\Gamma$ always  factors through the $K$-theory of the maximal $C^*$-algebra. For $K$-amenable groups for instance, $K_*(C^*_r\Gamma) \simeq K_*(C^*_m\Gamma)$ and the Baum-Connes conjecture is equivalent to the so-called maximal Baum-Connes conjecture. We shall mainly be interested in the present paper in the maximal Baum-Connes map and maximal Baum-Connes conjecture. This conjecture is false for some infinite property $(T)$ groups, it is  known to be true  for all torsion free discrete subgroups of $\SO (N, 1)$ or $\SU (N, 1)$ and also for all torsion free a-T-amenable groups which are $K$-amenable and for which the Baum-Connes conjecture was proved by Higson and Kasparov in \cite{HigsonKasparov}. As an example, the  maximal  Baum-Connes map is thus an isomorphism for all torsion free  discrete amenable groups. In the sequel, only the even dimensional case will be needed and all the statements are relative to the even maximal Baum-Connes map for torsion free groups:
$$
\mu_{\Gamma, \max}: K_0(B\Gamma) \longrightarrow K_0(C^*_m\Gamma).
$$

\medskip

\begin{theorem}\label{Absolute}
Let $M$ be a complete even-dimensional riemannian manifold and $D$ a generalized Dirac operator as before which is invertible near infinity. Assume that $\Gamma$ is torsion free and  the maximal Baum-Connes map  is  rationally onto. Then the following Atiyah-type formula holds for the Gromov-Lawson indices:
$$
\Ind_{(2)} (\tD^+) = \Ind (D^+).
$$
In particular, $\Ind_{(2)} (\tD^+)$ is an integer in this case.
\end{theorem}
\medskip

\begin{proof}
Since $\Gamma$ is torsion free, the rational maximal Baum-Connes map can be described as $\mu_\Gamma\otimes \Q: K^{\geo}_0(B\Gamma)\otimes \Q \rightarrow K_0(C^*_m\Gamma)\otimes \Q$, where $K_0^{\geo} (B\Gamma)$ is the geometric Baum-Douglas $K$-homology group. The higher  index class of the maximal completion of the generalized Dirac operator $\maD_m$ lives in $K_0(C^*_m\Gamma)$ and hence there exist classes $[Z_i, E_i, f_i]$ in the geometric $K$-homology group $K_0^{\geo} (B\Gamma)$ and rational numbers $q_i$ such that $\sum_i \mu_\Gamma [Z_i, E_i, f_i]\otimes q_i = \Ind (\maD_m)\otimes 1_\Q$. Recall that each $Z_i$ is a smooth closed $K$-oriented even dimensional manifold with a complex vector bundle $E_i$ over $Z_i$, and that $f_i:Z_i\to B\Gamma$ is a continuous map which classifies some smooth $\Gamma$-covering $\tZ_i\to Z_i$, and $[Z_i, E_i, f_i]$ is the class of the triple $(Z_i, E_i, f_i)$ with respect to some equivalence relations, see \cite{BaumDouglas}. Moreover, by definition, the map $\mu_\Gamma$ assigns to any triple $(Z, E, f)$ as above the higher index class in $K_0(C^*_m\Gamma)$ of the operator $\maD_m(Z, E)$ which is the maximal Dirac operator for the $K$-orientation twisted by some hermitian connection on $E$ and by the flat connection on the Michscheko bundle associated with the Galois covering given by $f$. Applying the regular and average traces to $\Ind (\maD_m)$ then gives
$$
\tau^{\reg}_* \Ind (\maD_m) = \sum_i q_i \tau^{\reg}_*\Ind (\maD_m (Z_i, E_i)) \text{ and }\tau^{\av}_* \Ind (\maD_m) = \sum_i q_i \tau^{\av}_*\Ind (\maD_m (Z_i, E_i)).
$$
We may then apply the classical Atiyah covering index theorem for each operator $\maD_m(Z_i, E_i)$ on the closed $K$-oriented manifold $Z_i$ and deduce the equality
$$
 \tau^{\reg}_*\Ind (\maD_m (Z_i, E_i)) = \tau^{\av}_*\Ind (\maD_m (Z_i, E_i), \quad \forall i.
$$
The conclusion follows.
\end{proof}

It seems to us that the following is an interesting problem.
\medskip

\begin{problem}
Assume that $\Gamma$ is torsion free.  Then is is true that for any Galois $\Gamma$-cover of complete even-dimensional riemannian manifolds $\tM\to M$ and any generalized Dirac operator $D$ which is invertible near infinity, one has: 
$$
\Ind_{(2)} (\tD^+) = \Ind (D^+) ?
$$ 
\end{problem}

\medskip

An affirmative answer implies    the weaker statement  that for torsion free groups one would have $\Ind_{(2)} (\tD)\in \Z$.
We end this paragraph with the corresponding statement for compact manifolds with boundary.

\medskip

\begin{theorem}\label{AtiyahAPS}
Let $\Gamma$ be a torsion free countable group such that the maximal Baum-Connes map for $\Gamma$ is rationally onto. Let $\pi_X: \tX\to X$ be any Galois $\Gamma$-covering of compact manifolds with boundaries.  We assume that $D_X$ is a generalized Dirac operator on $X$  which is a ''product operator''  in a neighborhood $(-\ep, 0]\times Y$ of the boundary $Y$  given as in \cite{APS1} by $\sigma(\pa/\pa t + D_Y)$, and  such that the generalized Dirac operator $D_Y$ on $Y$ is $L^2$ invertible. Then the $L^2$-index of the generalized Dirac operator $\tD$ with the global APS $\Gamma$-invariant boundary condition coincides with the integer index of the Dirac operator $D$ downstairs with the classical global APS boundary condition.
\end{theorem}
\medskip
So we assume in particular that $\pi_X$ restricts to a Galois $\Gamma$-covering of the boundaries $\pi_Y:\tY\to Y$, where $Y=\pa X$ and $\tY=\pa \tX$.
Recall from \cite{APS1} that the generalized Dirac operator $D$ with the global APS boundary condition has a well defined Fredholm index $\Ind^{APS} (D)\in \Z$. This result was generalized in \cite{Ramachandran} where the same result holds upstairs on $\tX$ for the $\Gamma$-invariant generalized Dirac operator $\tD$ subject to the $\Gamma$-invariant  boundary condition, which then has a well defined $L^2$ index $\Ind^{APS}_{(2)} (\tD)\in \R$. Therefore, under our assumptions on $\Gamma$, the statement of the theorem is that 
$$
\Ind^{APS}_{(2)} (\tD) = \Ind^{APS} (D).
$$

\begin{proof}
As in \cite{APS1}, we consider the smooth complete manifold $M$ obtained by attaching the semi-cylinder $Y\times [0, +\infty)$ to $X$ along its boundary. Then the same construction upstairs gives a smooth complete manifold $\tM$ and a Galois $\Gamma$-covering $\pi:\tM\to M$ which extends the Galois $\Gamma$-covering $\pi_X:\tX\to X$. The operator $D$ then ''extends'' to a generalized Dirac operator $D_M$ which is invertible near infinity in $M$, namely outside $X\subset M$. By \cite{APS1}[Proposition  3.11] and since the kernel of the Dirac operator on $Y$ is trivial in our case, we deduce that the APS-index $\Ind^{APS} (D)$ coincides with the Gromov-Lawson index of the operator $D_M$ on $M$. The $L^2$ index of the $\Gamma$-invariant operator $\widetilde{D_M}$  on $\tM$ is also well defined thanks to our Theorem \ref{Main}. Moreover, although more involved, it is still true that we have the equality
$$
\Ind^{APS}_{(2)} (\tD) = \Ind_{(2)} (\widetilde{D_M}),
$$
this is the main result proved in \cite{Vaillant}, see also \cite{Antonini}. Therefore, applying our Theorem \ref{Absolute}, we can conclude that
$$
\Ind^{APS}_{(2)} (\tD) = \Ind_{(2)} (\widetilde{D_M}) = \Ind (D_M) = \Ind^{APS} (D).
$$
\end{proof}

By using the APS index formula \cite{APS1} together with its $L^2$ version as proved in \cite{Ramachandran}, the previous theorem implies that the Cheeger-Gromov rho invariant of the generalized Dirac operator $\tD$ on any Galois covering $\tY\to Y$ vanishes when $\Gamma$ is torsion free with rationally onto full maximal Baum-Connes map. This result was actually proved  for general spin $\Gamma$ coverings $\tY\to Y$  with PSC and torsion free $\Gamma$ (not necessarily with some multiple which bounds a compact spin manifold) but under the stronger assumption that $\Gamma$ satisfies the maximal Baum-Connes conjecture, see \cite{PiazzaSchick} and also \cite{Keswani}. 
According to our results, it is easy to deduce that this full result (for torsion free groups)  should hold for all generalized Dirac operators $D$ such that $\maR >0$.

\medskip

\subsection{Some $L^2$ Gromov-Lawson PSC-invariants}\label{APSinvariants}

For any compact spin manifold $X$ with boundary $Y$  and for any $g\in \maR^+(Y)$ and any metric $G$ on $X$ which coincides with $g+dt^2$ on a collar neighborhood of $Y$ in $X$, we may extend $G$ to a complete metric $\widehat{G}$ on the smooth manifold $\widehat{X}$ obtained out of $X$ by attaching a semi-cylinder $(Y\times [0, +\infty), g+dt^2)$ at the boundary, see \cite{APS1}.  Then set as in \cite{GromovLawson}
$$
i(X, g):=\Ind (\widehat{D_{X}}^+),
$$
where $\widehat{D_{X}}$ is the associated Dirac operator on $\widehat{X}$. See \cite{GromovLawson} where it was proved that $\Ind (\widehat{D_{X}}^+)$ only depends on $X$ and $g$, this explains the notation. As explained in the end of the previous paragraph, it was proved in \cite{APS1}[Proposition 3.11] that  the invariant $i(X, g)$ coincides with the APS-index $\Ind^{APS} (D_X, B_Y)$ of  the Dirac operator $D^+_{X}$ on $X$ acting on $L^2$-spinors with the global APS boundary condition $B_Y$.  

In the same way, if $\tX\to X$ is a Galois $\Gamma$-cover of manifolds with boundaries $\tY$ and $Y$ so that we also have the Galois cover $\tY\to Y$. The Gromov-Lawson construction can then be done $\Gamma$-equivariantly and we obtain the invariant
$$
i_{(2)} (\tX, g) := \Ind_{(2)} \left(\widehat{D_{\tX}}^+\right),
$$
where $\widehat{D_{\tX}}$ is the $\Gamma$-invariant Dirac operator defined similarly upstairs and $\Ind_{(2)} \left(\widehat{D_{\tX}}^+\right)$ is its $L^2$ index obtained using our construction. Indeed, we have
\begin{lemma}
The index $\Ind_{(2)} (\widehat{D_{\tX}})$ does not depend on the choice of the metric $G$ on $X$.
\end{lemma}

\begin{proof}
Let us choose another metric $G'$ which coincides as well with $g+dt^2$  in a collar neighborhood of the bounday $Y$ through the diffeomorphic identification with $(-\ep, 0]\times Y$. The two extended metrics and spin structures on the augmented manifold $\widehat{X}$, then coincide outside a compact subspace ($X$ itself here), so that the identity map allows to apply our $L^2$ relative index Theorem \ref{relative} and deduce the formula
$$
\Ind_{(2)} \left(\widehat{D_{\tX, G}}^+\right) -\Ind_{(2)} \left(\widehat{D_{\tX, G'}}^+\right) = \int_X \widehat{A} (G) - \int_X \widehat{A} (G').
$$
The right hand side vanishes since it coincides with the index of the resulting generalized Dirac operator  on the double manifold $2X$. Notice that the metrics $G$ and $G'$ fit together to give a specific metric on the double, but the integral over $2X$ of the Atiyah-Singer form does not depend on the metric. 
\end{proof}

Following \cite{GromovLawson} again, we can introduce an $L^2$-invariant associated with any two metrics of PSC on a smooth closed odd dimensional spin manifold $N$, whose spin cobordism class is not necessarily torsion. More precisely, assume that $g_0$ and $g_1$ are two metrics with PSC on $N$ and consider the smooth  manifold $M=N\times \R$ endowed with any smooth  metric $G$ such that   $G=g_0+dt^2$ for $t\leq 0$ and $G= g_1+dt^2$ for $t\geq 1$. Then obviously the riemannian manifold $(M, G)$ is a spin manifold whose metric has PSC outside $N\times (0, 1)$ and hence near infinity. The invariant $i (N; g_0, g_1):=\Ind (D^+_{M, G})$ does not depend on the choice of metric $G$ for $0<t<1$. It was used  in \cite{GromovLawson} to deduce the existence of infinitely many noncordant metrics of PSC on the sphere $\S^7$. 

In the case where $N$ is not simply connected, one can still define the Gromov-Lawson index $i(N; g_0, g_1)$ but might use all normal covers $\tN\to N$ associated with normal subgroups of the fundamental group $\pi_1(N)$  and the corresponding Galois covers 
$$
\tM=\tN\times \R \longrightarrow M=N\times \R
$$ 
as we did before. More precisely, we introduce for any such Galois covering the invariant 
$$
i_{(2)} (\tN; g_0, g_1) := \Ind_{(2)} (D^+_{\tM, G}).
$$
Notice that if we can find such metric $G$ which has PSC on $N\times [0, 1]$ then for any such Galois cover, the $\Gamma$-invariant Dirac operator on $\tM$ will be $L^2$ invertible, and hence will have trivial $L^2$ index, i.e. in this case $i_{(2)} (\tN; g_0, g_1)=0$. This shows that  $i_{(2)} (\tN; g_0, g_1)$ is trivial when $g_0$ and $g_1$ belong to the same PSC concordance class, in particular when they are homotopic inside PSC metrics on $N$. \\
Recall that the space of metrics of PSC on $N$ is denoted by $\maR^+(N)$.
\medskip

\begin{proposition}\label{Properties}\
\begin{enumerate}
\item The invariant $i_{(2)} (\tN; g_0, g_1)$ does not depend on the metric $G$ on $N\times (0, 1)$.
\item One has $i_{(2)} (\tN; g_0, g_1) + i_{(2)} (\tN; g_1, g_2) + i_{(2)} (\tN; g_2, g_0) =0$ for any $g_i\in \maR^+(N)$. 
\item For any compact spin Galois cover $\tX\to X$ of manifolds with boundaries such that the boundary is the Galois cover $\tY\to Y$ as previously defined, we have for any $g_0, g_1\in \maR^+(Y)$:
$$
i_{(2)} (\tX, g_1) - i_{(2)} (\tX, g_0) = i_{(2)} (\tY, g_0, g_1).
$$ 
Hence, $i_{(2)} (\tX, g)$  only  depends on the concordance class of $g$ in $\maR^+(Y)$.
\end{enumerate}
\end{proposition}
\medskip

\begin{proof}\
Let $G'$ be another smooth metric on $M$ which coincides with $g_0+dt^2$ for $t\leq 0$ and with $g_1+dt^2$ for $t\geq 1$. Then the identity map $M\to M$ allows to apply our $L^2$ relative index formula to deduce that 
$$
\Ind_{(2)} \left(D^+_{\tM, G}\right) - \Ind_{(2)} \left(D^+_{\tM, G'}\right) = \int_{N\times [0, 1]} \widehat{A} (G) -  \int_{N\times [0, 1]} \widehat{A} (G').
$$ 
But the latter difference of integrals coincides with an integral over $N\times \S^1$ with $\S^1=[0, 2]/\sim$ and with the metric $G$ on $N\times [0, 1]$ and the leafwise metric $f^*G'$ on $N\times [1, 2]$ with $f: N\times [1, 2]\to N\times [0, 1]$ given by $f(y, t):= (y, 2-t)$. It is an obvious observation that we get in this way a well defined smooth  metric $\widehat{G}$ on the quotient $N\times \S^1$. Moreover,
$$
\int_{N\times [0, 1]} \widehat{A} (G) -  \int_{N\times [0, 1]} \widehat{A} (G')  =  \int_{N\times \S^1} \widehat{A} (\widehat{G}).
$$
The RHS being independent of the choice of metric, it  obviously vanishes, and we conclude that $\Ind_{(2)} (D^+_{\tM, G}) = \Ind_{(2)} (D^+_{\tM, G'})$.

Let us prove now the second item. We shall apply the $L^2$ $\Phi$-index theorem. More precisely, if we consider  two complete  metrics on $N\times \R$ defined as follows. The first one, denoted $G_{01}$ is any metric which coincides with $g_0+dt^2$ on $N\times (-\infty, 0]$ and with $g_1+dt^2$ on $N\times [1, +\infty)$. The second, denoted $G_{02}$ is any metric which coincides again with $g_0+dt^2$ on $N\times (-\infty, 0]$ but with $g_2+dt^2$ on $N\times [1, +\infty)$.  These corresponding two complete $\Gamma$-covers then are identified on $\tN\times (-\infty, 0]$ by the identity map, and we may apply the $\Phi$-relative $L^2$-index theorem since the metrics $G_{01}$ and $G_{02}$ both have PSC near infinity, we get
$$
i_{(2)} (\tN; g_0, g_1)  - i_{(2)} (\tN; g_0, g_2) = i_{(2)} (\tN; g_2, g_1),
$$
and therefore since $i_{(2)} (\tN, g, g') + i_{(2)} (\tN, g', g)=0$ by the sam argument:
$$
0=i_{(2)} (\tN; g_0, g_1)  - i_{(2)} (\tN; g_0, g_2) - i_{(2)} (\tN; g_2, g_1) = i_{(2)} (\tN; g_0, g_1) + i_{(2)} (\tN; g_2, g_0) + i_{(2)} (\tN; g_1, g_2).
$$

It remains to  prove the third item. Let $G$ and $G'$ be smooth metrics on $X$ which coincide on a collar neighborhood $U_\ep\simeq Y\times [-\ep, 0]$ with $g_0+dt^2$ and $g_1+ dt^2$ respectively. We are assuming again that $\tX\to X$ is also a Galois $\Gamma$-cover over $X$ which restricts to $\tY\to Y$ at the boundaries. Consider again the augmented  manifold $\widehat{\tX}$ obtained by adding the semi-cylinder $\tY\times [0, +\infty)$ with the corresponding metric lifted from $Y\times [0, +\infty)$. We obtain our invariants  $i_{(2)} (\tX, g_1)$ and  $i_{(2)} (\tX, g_0)$ as before.  We consider the new  metric $G''$ on $X$ which is defined as follows. On $X\smallsetminus \overline{U_{\ep/2}}\simeq Y\times [-\ep/2, 0]$ we take $G''=G'$. On $U_{\ep/2}$ we take a smooth leafwise metric  which coincides with $g_1+dt^2$ near $-\ep/2$ and with $g_0+dt^2$ near $0$. Using the same construction as before, we get the invariant $i_{(2)} (X, g_0)$ by using this new leafwise metric. Recall that $i_{(2)} (X,  g_0)$ does not depend on the choice of such metric. By applying the $\Phi$-relative index theorem \ref{PhiRelative} to the $\Gamma$-covering manifold $\widehat{\tX}\to \widehat{X}$ with the metric associated with $G'$ on the one hand, and the covering  $\tY\times \R \to Y\times \R$ with the metric associated with  $G$ defining the invariant $i_{(2)} (Y; g_0, g_1)$ on the other hand, we get
$$
i_{(2)} (\tX, g_1) - i_{(2)} (\tY, g_0, g_1) = i_{(2)} (\tX, g_0) .
$$
Hence the conclusion. 

Hence we see from this relation that the difference $i_{(2)} (\tX, g_1) - i_{(2)} (\tX, g_0) $ does not depend on the covering bordant manifold $\tX\to X$. Moreover, if $g'_0$ belongs to the concordance class of $g_0$ then we know that the invariant $i_{(2)} (Y; g'_0, g_0)$ can be computed as the $\Gamma$-index of the Dirac operator on the $\Gamma$-cover $\tY\times \R\to Y\times \R$ for a metric which has PSC everywhere, and hence $i_{(2)} (Y; g'_0, g_0)=0$. Therefore, we deduce that $i_{(2)} (\tX, g_0)$  only depends on the concordance class of $g_0$ in $\maR^+(Y, F_Y)$.

\end{proof}

We have the following corollary of Theorem \ref{AtiyahAPS}.

\medskip

\begin{corollary}\label{Atiyah-GL}
Assume that the group $\Gamma$ is \underline{torsion free} and that the maximal Baum-Connes map for $\Gamma$ is rationally onto, then  the following Atiyah-type theorem holds:
$$
i_{(2)} (\tN; g_0, g_1) = i (N; g_0, g_1), \quad \forall g_0, g_1\in \maR^+(N).
$$
\end{corollary}
\medskip

\begin{proof}\ Just apply Theorem \ref{AtiyahAPS}  to the $\Gamma$-covering $\tN\times [0, 1]$ of the compact manifold with boundary $N\times [0, 1]$. 
\end{proof}


When $\Gamma$ has torsion, this corollary is  false. Indeed, take the Lens space $L=\S^7/\Z_n$ where $\Z_n$ acts diagonally on $\S^7\subset \C^4$. Then if we fix the standard metric $h_0$ on $L$  which is induced by the standard metric $g_0$ on $\S^7$ with constant scalar curvature equal to $1$, then the map
$$
\kappa : \pi_0(\maR^+(L))  \longrightarrow \Z \text{ defined by } \kappa (g):=i(\S^7; g_0, g) - n\times i(L; h_0, g),
$$
has infinite range in $\Z$. Notice that $\kappa (g)=n\times  \left(i_{(2)} (\S^7; h_0, g) - i (L; h_0, g)\right)$. To see this, just observe by the APS formula that we have  the following relation:
$$
\kappa (g) = n \times  \frac{\rho (g_0) - \rho (g)}{2},
$$
where $\rho$ is the APS rho invariant, a difference of eta invariants \cite{BenameurPiazza}. Then the main result proved in  \cite{BotvinnikGilkey} completes the argument.
More generally, using the results of \cite{BotvinnikGilkey}, one can deduce that any $4k+3$-dimensional spin manifold, with $k\geq 1$ which has a metric of PSC and finite fundamental group, admits an infinite collection of non-concordant metrics of PSC, parametrized by a similar invariant $\kappa$ as defined above. 


\begin{remark}
By using Proposition 2.14 in \cite{KreckStolz}, we see that $g\mapsto i(L; h_0, g)$ already has infinite range.
\end{remark}

Given an orientation preserving diffeomorphism $F\in \Diff^+(N)$ of $N$, we may transport any PSC metric $g$ into the PSC metric $F^*g$. Notice also that the spin structure associated with $g$ on $M$ canonically  determines a spin structure associated with $F^*g$ which is in fact the pull-back spin structure under $F$.  By applying again the APS formula as well as its $L^2$ version, we can deduce the following

\medskip
  
\begin{theorem}
Let $F\in \Diff^+(N)$ be an orientation preserving diffeomorphism. Then for any Galois cover $\tN\to N$ with group  $\Gamma$, we have: 
$$
i_{(2)} (\tN; g_0, F^*g_0) = i(N; g_0, F^*g_0).
$$
In particular, the invariant $\kappa_{g_0}^\Gamma: g\mapsto i_{(2)} (\tN; g_0, g) - i(N; g_0, g)$ induces a well defined map
$$
\kappa_{g_0}^\Gamma : \pi_0 \left( \maR^+(N)/\Diff^{+} (N) \right) \longrightarrow \R.
$$
\end{theorem}
\medskip

\begin{proof}
The first part of the theorem is an easy consequence of the APS theorem. 

Assume now that $g$ is a given PSC metric on $N$, then  by Proposition \ref{Properties}
$$
i(N; g_0, F^*g) - i(N; g_0, g)  = i(N; g, F^*g)\text{ while }i_{(2)} (\tN; g_0, F^*g) - i_{(2)} (\tN; g_0, g)  = i_{(2)}(\tN; g, F^*g).
$$
Therefore, $\kappa_{g_0}^\Gamma (F^*g) - \kappa_{g_0}^\Gamma (g) = i_{(2)}(\tN; g, F^*g) - i(N; g, F^*g) = 0.$
\end{proof}

Assume now  that $\tN\to N$ is the universal cover, so with $\Gamma=\pi_1(N)$. A smooth orientation preserving diffeomorphism  $F\in \Diff^+(N)$ of $N$ as above then lifts to a smooth orientation preserving diffeomorphism $\tF$ of $\tN$ which is equivariant relative to an outer automorphism  of the group $\Gamma$, and we can choose for the given $F$ a representative automorphism that we denote by $\phi_F\in \Aut (\Gamma)$, which allows to define an action of $\Z$ on the group $\Gamma$. We denote by $\Gamma\rtimes \Z$ the semi-direct product, see for instance \cite{KreckStolz, XieYu}. Notice that we could as well consider the transported metric $\tF^*\tg$ of the pull-back metric $\tg$ on $\tN$ but it is an obvious observation that  the new metric $\tF^*\tg$ is still  $\Gamma$-invariant and induces  the metric $F^*g$.

\begin{proposition}
Let $\tN\to N$ be the universal cover of the closed riemannian manifold $(N, g)$ with PSC. Denote by $\widehat{N}$ the closed manifold which is the quotient under the free proper action of $\Gamma\rtimes \Z$, i.e. $\widehat{N}:= (\tN\times \R)/\Gamma\rtimes \Z$. Then, using the structures induced from those fixed on $N$, the index of the Dirac operator on $\widehat{N}$ coincides with the Gromov-Lawson index $ i (N; g, F^*g)$.
\end{proposition}

\begin{proof}\ We consider a metric $G$ on $N\times [0, 1]$ which coincides with $g$ near $0$ and with $F^*g$ near $1$, then this metric induces a smooth metric on the mapping torus $\widehat{N}$, and we have
$$
\int_{N\times (0, 1)} \widehat{A} (G) = \int_{N_F} \widehat{A} (T\widehat{N}).
$$
Hence by \cite{APS1}[Proposition 3.11], we may compute $i(N; g, F^*g)$ using the  Atiyah-Patodi-Singer index formula on $N\times [0, 1]$, i.e.
$$
i(N; g, F^*g) = \int_{N\times (0, 1)} \widehat{A} (G) -\frac{\eta (D_{(N, F^*g)}) - \eta (D_{(N, g)})}{2}.
$$
But  we know that $\eta (D_{(N, F^*g)}) = \eta (D_{(N, g)})$, hence we get
$$
i(N; g, F^*g) = \int_{N\times (0, 1)} \widehat{A} (G) = \int_{\widehat{N}} \widehat{A} (T\widehat{N}).
$$
The proof is complete since the latter integral is equal to the index of the Dirac operator by the Atiyah-Singer formula on $\widehat{N}$. 
\end{proof}

In \cite{KreckStolz}, the similar previous construction allowed to deduce the existence  for any $k\geq 2$ of a  closed spin manifold $N$ of dimension $4k-1$, which moreover admits  a metric of PSC, such that the image of the map 
$$
\Diff^{\Spin} (N) \longrightarrow \Z
$$
sending $F$ to the index of the Dirac operator on the mapping torus $\widehat{N}$, is $\Z$ for $k$ even and $2\Z$ for $k$ odd. Here $\Diff^{\Spin} (N)\subset \Diff^+(N)$ is the subgroup of spin-preserving diffeomorphisms.

\medskip

\appendix

\section{A Rellich lemma and a second proof of Theorem \ref{Main}}\label{Rellich}

We give in this appendix  another proof of Theorem \ref{Main} based on the Rellich lemma and following the direct method applied in \cite{XieYu} for the higher index theory. It is worthpointing out that the present proof does not use Theorem \ref{Compatibility} since this latter used in its proof Theorem \ref{Main}. We have  added a short proof of the Rellich lemma in this von Neumann context, which is ready for generalizations for instance to foliations. Given a compact subspace $K$ of $M$, we denote by $L^2_{s, K} (\tM, \tS)$ the subspace of $L^2_s (\tM, \tS)$ composed of those sections  which are supported within the $\Gamma$-compact subspace $\pi^{-1} (K)$ of $\tM$. Recall from Section \ref{Preliminaries} the semi-finite von Neumann algebra $\maM=L^2(\tM, \tS)^\Gamma$ with its trace $\tau$. We have again chosen a fundamental domain $F$ for the Galois cover $\pi:\tM\to M$ with the usual properties.

\begin{lemma}\label{Rellich} ($L^2$ Rellich lemma)\
Assume that for a fixed compact subspace $K$ of $M$, an element $T\in \maM$ factors through the inclusion $L^2_{1, K} (\tM, \tS) \hookrightarrow L^2 (\tM, \tS)$ via $\Gamma$-invariant operators. Then $T$ belongs to the ideal $\maK (\maM, \tau)$ of $\tau$-compact operators.
\end{lemma}

\begin{remark} When $\Gamma$ is trivial,  this is the exact statement of the classical Rellich lemma on $M$.
\end{remark}

\begin{proof}
Let $(f_n)_{n\geq 0}$ be a uniformly bounded sequence which belongs to $L^2_{1, K} (\tM, \tS)$ and let $f\in L^2_{1, K} (\tM, \tS)$ be such that for any $\varphi\in L_{1, K}^2(M, S)\simeq L_{1,K}^2(F, \tS)$, we have
$$
\sum_{\gamma\in \Gamma} \langle \gamma^*(f_n - f), \varphi\rangle_{L_1^2(F, \tS)} \longrightarrow  0\text{ as }n\to +\infty. 
$$
Then for any $\gamma\in \Gamma$, the restriction of $\gamma^*(f_n-f)$ to $F$ defines a sequence in $L^2_{1,K} (M, S)$ which converges weakly to $0$ in that Hilbert space. Since $K$ is compact in $M$, the classical Rellich lemma in $M$ applies and we can deduce that the sequence $(\gamma^*f_n)_{n\geq 0}$ converges to $\gamma^*f$ in $L^2 (M, S)\simeq L^2(F, \tS)$, say
$$
\vert\vert f_n - f\vert\vert_{L^2 (\gamma F, \tS)} \longrightarrow  0\text{ as }n\to +\infty. 
$$
Since 
$$
\sum_{\gamma\in \Gamma}  \vert\vert f_n - f\vert\vert_{L^2 (\gamma F, \tS)} = \vert\vert f_n - f\vert\vert_{L^2(\tM, \tS)} <+\infty,
$$
an easy argument shows that the sequence $(f_n)_{n\geq 0}$ converges in $L^2(\tM, \tS)$ to $f$. 

Suppose now that $T\in \maM\subset B(L^2(\tM, \tS))$ satisfies the assumptions of the lemma, i.e. $T= \iota \circ A$ for some $\Gamma$-equivariant operator $A:L^2(\tM, \tS) \to L^2_{1, K} (\tM, \tS)$. If a given uniformly bounded sequence  $(\sigma_n)_{n\geq 0}$ from $L^2(\tM, \tS)$ converges weakly relative to $\maM$ in the sense of \cite{Kaftal82} to some section $\sigma\in L^2(\tM, \tS)$, then for any $\psi\in L^2(M, S)$ with $\vert\vert\psi\vert\vert =1$, and using the $\tau$ rank one operator which is the orthogonal projection onto the line generated by $\psi$ tensorized by the identity of $\ell^2\Gamma$, we deduce that
$$
\sum_{\gamma\in \Gamma} \langle g^* A (\sigma_n - \sigma), A\psi \rangle_{L^2_{1, K} (F, \tS)} = \sum_{\gamma\in \Gamma} \langle g^*(\sigma_n - \sigma), A^*A\psi \rangle_{L^2 (F, \tS)} \longrightarrow  0\text{ as }n\to +\infty. 
$$ 
We deduce from the previous discussion that $\left((\iota\circ A) (\sigma_n)\right)_{n\geq 0}$ converges in $L^2(\tM, \tS)$ to $(\iota\circ A) (\sigma)$, and hence that $T$ is $\tau$-compact, using again  the compactness criterion of \cite{Kaftal82}.
\end{proof}

Recall that the generalized Dirac operator $D$ satisfies the generalized Lichnerowicz formula 
$$
D^2=\nabla^*\nabla +\maR,
$$
and that we assume that there exist a compact subspace $K$ of $M$ such that $\maR\vert_{M\smallsetminus K}$ is uniformly positive. Then we have fixed a compactly supported smooth real valued function $\rho\in C_c^\infty (M)$ and a constant $c>0$ such that 
$$
\maR \geq (c-\rho^2)\Id \text{ over } M.
$$ 
 The operator $(\tD^2+\rho^2)^{1/2}$ is then a well defined self-adjoint non-negative operator.

\begin{lemma}
The operator $(\tD^2+\rho^2)^{1/2}$ is injective with bounded inverse. Moreover,  its inverse, denoted  $(\tD^2+\rho^2)^{-1/2}=:\tT$, is a bounded operator from $L^2 (\tM, \tS)$ to the Sobolev space $L_1^2(\tM, \tS)$. In particular, $\tT$ belongs to the von Neumann algebra $\maM$.
\end{lemma}

\begin{proof}
Let $\sigma\in \Dom (\tD^2)$. Then 
$$
\langle (\tD^2+\rho^2)\sigma, \sigma\rangle = \langle (\widetilde{\nabla}^*\widetilde{\nabla}  + \maR +\rho^2\Id) \sigma, \sigma\rangle\geq  \vert\vert \widetilde{\nabla} \sigma \vert\vert^2 + c\vert\vert \sigma\vert\vert^2 \geq c\vert\vert \sigma\vert\vert^2.
$$
Hence the self-adjoint operator $\tD^2+\rho^2$ is injective and has  a bounded inverse which obviously commutes with $\Gamma$ and belongs to the von Neumann algebra $\maM$. Indeed, its spectrum is contained in $[c, +\infty)$, so that the spectrum of $(\tD^2+\rho^2)^{1/2}$ is contained in $[\sqrt{c}, +\infty)$. Therefore, the operator $\tT$ is a well defined bounded non-negative operator on $L^2(\tM, \tS)$. Moreover, for $\sigma\in L^2 (\tM, \tS)$ we can write
\begin{eqnarray*}
\vert\vert (\tD^2+\rho^2)^{-1/2} \sigma\vert\vert_1^2  & = & \langle (\tD^2+\rho^2)^{-1} \sigma, \sigma \rangle  + \langle (\tD^2+\rho^2)^{-1/2} \widetilde\nabla^*\widetilde\nabla (\tD^2+\rho^2)^{-1/2} \sigma, \sigma\rangle \\
& \leq  & \frac{1}{c}  \vert\vert\sigma\vert\vert^2 + \langle (\tD^2+\rho^2)^{-1/2} [(\tD^2+\rho^2)-(\rho^2+\maR)] (\tD^2+\rho^2)^{-1/2} \sigma, \sigma\rangle
 \end{eqnarray*}
Notice that by the properties of the continuous functional calculus, we deduce that 
$$
(\tD^2+\rho^2)^{-1/2} (\tD^2+\rho^2) (\tD^2+\rho^2)^{-1/2} = \Id.
$$
Hence, we obtain since $\rho^2 \Id+\maR \geq c\Id$:
$$
\vert\vert (\tD^2+\rho^2)^{-1/2} \sigma\vert\vert_1^2  \leq (1+\frac{1}{c}) \vert\vert\sigma\vert\vert^2.
$$
Therefore, the operator $(\tD^2+\rho^2)^{-1/2}$ is a bounded operator from $L^2(\tM, \tS)$ to $L^2_1(\tM, \tS)$ as announced.
\end{proof}

\begin{remark}
The same proof shows that for any $\lambda\in \R$, the operator $T_\lambda:=(\tD^2+\rho^2+\lambda^2)^{-1/2}$ extends to a bounded operator from $L^2(\tM, \tS)$ to $L_{1}^2(\tM, \tS)$.
\end{remark}

The following proposition is a von Neumann version of results from \cite{XieYu}.

\begin{proposition}\label{Flambda}
For any $\lambda\in \R$, consider the operator $\tF_\lambda:=\tD(\tD^2+\rho^2+\lambda^2)^{-1/2}=\tD\tT_\lambda$, then:
\begin{enumerate}
\item $\tF_\lambda$  extends to a bounded operator on $L^2(\tM, \tS)$ which belongs to the von Neumann algebra $\maM$;
\item The $\Gamma$-invariant  operator $\tQ=(\tD^2+\rho^2)^{-1/2} \tF_0$ is a self-adjoint parametrix for $\tD$ modulo   $\tau$-compact operators with respect to the von Neumann algebra $\maM$.
\end{enumerate}
\end{proposition}

\begin{proof} 
We  notice again by the properties of the continuous functional calculus that
$$
\tT_\lambda (\tD^2+\rho^2+\lambda^2) \tT_\lambda = \Id,
$$
say is densely defined and extends to a bounded operator which is the identity operator. Therefore since $\rho^2+\lambda^2$ is bounded on $\tM$, we deduce that the operator $\tT_\lambda \tD^2\tT_\lambda$ also extends to a bounded operator on $L^2(\tM, \tS)$. Hence we deduce by standard arguments that the operator $\tF_\lambda$ also extends to a bounded operator on $L^2(\tM, \tS)$. 
On the other hand the $\Gamma$-invariance is obvious by construction since $\rho$ is pulled-back from $M$, hence $\tF_\lambda\in \maM$.

 We now prove that $\tF=\tF_0$ is Fredholm with respect to $\maM$ with quasi-inverse given by $\tF$ itself, this will finish the proof. 
 Notice that we can use the following integral expression where   one easily checks that the integral converges in the operator norm:
$$
\tT=\frac{2}{\pi} \int_0^{+\infty} (\tD^2+\rho^2+\lambda^2)^{-1} d\lambda.
$$
This allows to see for instance that the commutator $[\tD, \tT]$ extends to a bounded operator defined by the formula:
$$
[\tD, \tT]=\frac{2}{\pi} \int_0^{+\infty} (\tD^2+\rho^2+\lambda^2)^{-1} [\rho^2, \tD] (\tD^2+\rho^2+\lambda^2)^{-1} \, d\lambda,
$$
and the operator $[\rho^2, \tD]$ being bounded while $\tD^2+\rho^2\geq c\Id$, we can deduce the
$$
\vert\vert (\tD^2+\rho^2+\lambda^2)^{-1} [\rho^2, \tD] (\tD^2+\rho^2+\lambda^2)^{-1}\vert\vert \leq \vert\vert [\rho^2, \tD]\vert\vert \frac{1}{(c+\lambda^2)^2}.
$$
Moreover, 
$$
\tF^2 = \tD\tT\tD\tT=\tD^2 \tT^2 - \tD [\tD, \tT] \tT=\Id - \rho^2 \tT^2 - \tD [\tD, \tT] \tT.
$$
Similarly, we can write
\begin{eqnarray*}
\tD [\tD, \tT]  & = & \frac{2}{\pi} \int_0^{+\infty} \tD [\tD, (\tD^2+\rho^2+\lambda^2)^{-1}] d\lambda\\
& = &  \frac{2}{\pi} \int_0^{+\infty} \tD (\tD^2+\rho^2+\lambda^2)^{-1} [\rho^2, \tD] (\tD^2+\rho^2+\lambda^2)^{-1} d\lambda
\end{eqnarray*}
Since $\rho$ and $[\tD, \rho]$ are both $\Gamma$-compactly supported zero-th order $\Gamma$-invariant operators, the operator 
$$
[\rho^2, \tD] (\tD^2+\rho^2+\lambda^2)^{-1}
$$ 
is a compact operator with respect to $\maM$. Indeed,  the operator $(D^2+\rho^2+\lambda^2)^{-1}$ is bounded below from $L^2(\tM, \tS)$ to $L_1^2(\tM, \tS)$. On the other hand applying $[\rho^2, \tD]$ sends $L_1^2(M, S; \mu)$ to the subspace $L_{1, \Supp (\rho)}^2(M, S; \mu)$ of $L_1^2(M, S; \mu)$. Now the Rellich lemma \ref{Rellich}  allows to conclude that  $[\rho^2, \tD] (\tD^2+\rho^2+\lambda^2)^{-1}$ is compact with respect to the von Neumann algebra $\maM$. By the first item, the operator $\tD (\tD^2+\rho^2+\lambda^2)^{-1}$ is bounded so that for any $\lambda \geq 0$ the operator 
$$
\tD (\tD^2+\rho^2+\lambda^2)^{-1} [\rho^2, \tD] (\tD^2+\rho^2+\lambda^2)^{-1}
$$
is compact with respect to the von Neumann algebra $\maM$. By the operator norm convergence of the integral, we deduce that $\tD [\tD, \tT]$ is compact with respect to the von Neumann algebra $\maM$. Finally, the same argument again allows to prove that the operator $\rho^2 \tT^2$ is also compact with respect to $\maM$ so that the proof is complete.

\end{proof}


\begin{proof} (of Theorem \ref{Main})

The self-adjoint $\Gamma$-invariant operator $\tQ=(\tD^2+\rho^2)^{-1/2} \tD (\tD^2+\rho^2)^{-1/2}$ is an odd for the grading parametrix for $\tD$ modulo compact operators with respect to  $\maM$. Indeed, we have proved that $\tF^2-I$  is a compact operator with respect to the von Neumann algebra $\maM$, hence so is its adjoint, say:
$$
\Id - \tD\tQ = \tS \in \maK (\maM, \tau_\Gamma)\text{ and } \Id - \tQ\tD = \tS^* \in \maK (\maM, \tau_\Gamma).
$$
Hence, if we denote as before by  $P_\ep$ the orthogonal projection given by the functional calculus $P_\ep:=1_{[0, \ep]} (\tD^2)$, we can write
$$
P_\ep -  P_\ep \tQ(\tD P_\ep) =  P_\ep \tS^* P_\ep.
$$
If $\sigma\in \Im (P_\ep)$ then we thus get
$$
\vert\vert P_\ep \tS^* \sigma\vert \vert \geq (1- \sqrt{\ep} \vert\vert \tQ\vert \vert )\vert\vert \sigma \vert\vert.
$$
Therefore, for $\ep>0$ small enough, the operator $P_\ep \tS^*P_\ep$ restricted to the image of $P_\ep$ is injective with closed range, and is therefore an isomorphism between $\Im (P_\ep)$ and $\Im (P_\ep \tS^*P_\ep)$. Since $P_\ep \tS^*P_\ep$ is compact with respect to the von Neumann algebra $P_\ep \maM P_\ep$ due to the same property for $\tS^*$ with respect to $\maM$, we deduce that $P_\ep$ must be compact with respect to $P_\ep \maM P_\ep$, that is that the projection $P_\ep$ is compact with respect to $\maM$. We conclude again that $P_\ep$ has finite $\tau$-trace as allowed since any $\tau$-compact projection is $\tau$-finite. 
\end{proof}

\bigskip

\bibliographystyle{alpha}

\end{document}